\documentclass[10pt,a4paper]{article}

\usepackage[T1]{fontenc}
\usepackage{amssymb}
\usepackage{amsthm}
\usepackage{mathtools}
\usepackage{enumitem}
\usepackage{aliascnt}
\usepackage{hyperref}
\numberwithin{equation}{section}
\newtheorem{theorem}{Theorem}[section]

\newaliascnt{proposition}{theorem}
\newtheorem{proposition}[proposition]{Proposition}
\aliascntresetthe{proposition}

\newaliascnt{lemma}{theorem}
\newtheorem{lemma}[lemma]{Lemma}
\aliascntresetthe{lemma}

\newaliascnt{corollary}{theorem}

\aliascntresetthe{corollary}

\newaliascnt{conjecture}{theorem}
\newtheorem{conjecture}[conjecture]{Conjecture}
\aliascntresetthe{conjecture}

\theoremstyle{definition}
\newaliascnt{claim}{theorem}
\newtheorem{claim}[claim]{Claim}
\aliascntresetthe{claim}

\usepackage[noabbrev,capitalize]{cleveref}
\crefname{equation}{}{}

\crefname{theorem}{Theorem}{Theorems}
\crefname{proposition}{Proposition}{Propositions}
\crefname{lemma}{Lemma}{Lemmas}
\crefname{corollary}{Corollary}{Corollaries}
\crefname{conjecture}{Conjecture}{Conjectures}
\crefname{claim}{Claim}{Claims}

\newcommand{\F}{\mathcal F}
\newcommand{\G}{\mathcal G}
\newcommand{\X}{\mathcal X}
\newcommand{\M}{\mathcal M}
\newcommand{\Rcal}{\mathcal R}

\newcommand{\cl}{\operatorname{cl}}

\title{Extremal Families for the Erd\H{o}s--Kleitman Problem: The Missing Constructions}
\author{
    Cheng Chi\thanks{School of Mathematical Sciences, Shanghai Jiao Tong University, 800 Dongchuan Road, Shanghai 200240, China.
    Email: chengchi@sjtu.edu.cn.
    Supported by National Key R\&D Program of China under grant No. 2022YFA1006400 and National Natural Science Foundation of China No. 12571376.
    } \qquad Yan Wang\thanks{School of Mathematical Sciences, Shanghai Jiao Tong University, 800 Dongchuan Road, Shanghai 200240, China.
    Email: yan.w@sjtu.edu.cn.
    Supported by National Key R\&D Program of China under grant No. 2022YFA1006400, National Natural Science Foundation of China under grant No. 12571376.
    } }
\date{}

\begin{document}
\maketitle

\begin{abstract}
        For integers $n\ge s\ge2$, let $e(n,s)$ be the maximum size of a family $\F\subseteq2^{[n]}$ with no $s$ pairwise disjoint members.
        The problem of determining $e(n,s)$, now called the Erd\H{o}s--Kleitman problem, is closely related to the well-known Erd\H{o}s matching problem.
        Frankl and Kupavskii posed a meta-conjecture predicting that the maximum is always attained by a weighted family.

        Fix $m\ge3$, write $n=ms+c$ with $0\le c<s$, and set $\ell=s-c$.
        For $0\le k\le m$, let $a_k=ms-kc-1$.
        For $A\in\binom{[n]}{a_k}$, define
        \[
                \mathcal H^k(m,s,\ell;A)\coloneqq
                \{F\subseteq[n]: k|F|+|F\cap A|\ge m(k+1)\}.
        \]
        This defines a unified class of weighted families with matching number less than $s$.
        Among these families, $\mathcal H^0$, $\mathcal H^1$, and $\mathcal H^m$ were previously known to be extremal in different ranges of $c$.

        We show that for $1\le k\le m-1$, all families $\mathcal H^k$ are uniquely extremal in some ranges of $c$.
        More precisely, we prove that for every $m\ge3$ and every $1\le k\le m-1$, there exist constants $\alpha=\alpha(m,k)>0$, $\beta=\beta(m,k)>0$ and an integer $s_0=s_0(m,k)$ such that, for all integers $s\ge s_0$ and all integers $c$ with $0\le c<s$, the only extremal families for $e(n,s)$ are the families $\mathcal H^k(m,s,\ell;A)$ with $A\in\binom{[n]}{a_k}$ whenever $\beta s^{(k-1)/k}\le c\le \alpha s^{k/(k+1)}$.
        In particular, this result determines an infinite number of new extremal families for the Erd\H{o}s--Kleitman problem and verifies the Frankl--Kupavskii meta-conjecture in these ranges.
        This also provides a quantitative extension of the result of Kupavskii and Sokolov on the extremality of $\mathcal H^1$.
\end{abstract}

\section{Introduction}\label{sec:intro}

Let $[n]\coloneqq\{1,2,\ldots,n\}$.
For a set $X$ and an integer $k$, write $\binom Xk$ for the family of all $k$-element subsets of $X$, and write $\binom X{\ge k}\coloneqq\bigcup_{r\ge k}\binom Xr$ and $\binom X{<k}\coloneqq\bigcup_{0\le r<k}\binom Xr$.
A \emph{matching} in a family is a collection of pairwise disjoint sets.
An $s$-\emph{matching} is a matching of size $s$.
Given a family $\F$, its \emph{matching number} $\nu(\F)$ is the size of a largest matching contained in $\F$.

The Erd\H{o}s--Kleitman problem asks for the maximum size of a family of subsets of $[n]$ with no $s$ pairwise disjoint members.
Equivalently, for integers $n\ge s\ge2$, define
\[
        e(n,s)\coloneqq
        \max\{ |\F|:\F\subseteq2^{[n]}\text{ and }\nu(\F)<s\}.
\]
The Erd\H{o}s--Kleitman can be seen as the non-uniform counterpart of the well-known Erd\H{o}s Matching Conjecture \cite{ErdosEMC}.
For $n\ge rs$, the conjecture asserts that
\[
        \max\left\{|\mathcal G|:\mathcal G\subseteq\binom{[n]}r
        \text{ and }\nu(\mathcal G)<s\right\}
        =
        \max\left\{
        \binom{rs-1}r,
        \binom nr-\binom{n-s+1}r
        \right\}.
\]
The two terms on the right arise, respectively, from taking all $r$-sets on an $(rs-1)$-element ground set and from taking all $r$-sets that meet a fixed $(s-1)$-set.
For $r=2$, it is the theorem of Erd\H{o}s and Gallai \cite{ErdosGallai}, while for $s=2$, it reduces to the Erd\H{o}s--Ko--Rado theorem \cite{ErdosKoRado}.
The conjecture is known for $r=3$ by Frankl \cite{FranklDAM}, following earlier work of \L{}uczak and Mieczkowska \cite{LuczakMieczkowska}.
For arbitrary $r$, Frankl and Kupavskii \cite{FranklKupavskiiConcentration} substantially extended the range in which the family of all $r$-sets meeting a fixed $(s-1)$-set is extremal. They also established a Hilton--Milner-type stability theorem \cite{FranklKupavskiiTwoProblems}.
More recently, Frankl, Lu, Ma and Wu \cite{FranklLuMaWu} proved the conjecture for $r=4$ in the range $n\ge5(s-1)$ when $n$ is sufficiently large, using a stability theorem for shifted $4$-uniform families.
In a subsequent preprint, Hou, Hu and Liu \cite{HouHuLiuEMC} proved the conjecture for $r=4$ and $n\ge4s$ whenever $s\ge6962$.
Nevertheless, the conjecture remains open in general for every $r\ge 4$.
Structural theories for perfect and almost-perfect matchings in dense uniform hypergraphs likewise emphasize the role of space and divisibility barriers; see Keevash and Mycroft \cite{KeevashMycroft} and the survey of Keevash \cite{KeevashSurvey}.
Matchings also arise naturally through transversals.
Keevash, Pokrovskiy, Sudakov and Yepremyan \cite{KeevashPokrovskiySudakovYepremyan} proved that every Steiner triple system of order $n$ has a matching of size $n/3-O(\log n/\log\log n)$, and Montgomery \cite{MontgomeryRBS} later improved this bound to at least $(n-4)/3$ for all sufficiently large $n$.
For a systematic account of the correspondence between Latin transversals and matchings in $3$-partite $3$-uniform hypergraphs, see Montgomery \cite{MontgomerySurvey}.

More broadly, the Erd\H{o}s--Kleitman problem belongs to the study of extremal families in the Boolean cube.
Fundamental directions in this area include shadow and stability problems \cite{KeevashShadows}, vertex isoperimetry in the cube \cite{KeevashLongCube}, and restricted intersection problems for uniform and non-uniform set systems \cite{AnsteeKeevash,KeevashSudakovCross}, together with vector-valued analogues \cite{KeevashLongIntersections}.
Related forbidden-configuration problems include set systems without simplices or clusters \cite{KeevashMubayiSimplex} and sunflower problems for set systems of bounded dimension \cite{FoxPachSukSunflowers}.
Although the forbidden configurations differ, all these problems seek quantitative bounds on families satisfying prescribed restrictions, and several also determine exact extremal constructions or establish stability.

We now return to the Erd\H{o}s--Kleitman problem.
Kleitman determined $e(n,s)$ when $n\equiv-1\pmod s$ or $n\equiv0\pmod s$ \cite{Kleitman}.
Further direct results for non-uniform families with bounded matching number include inequalities of Frankl and Kupavskii \cite{FranklKupavskiiInequalities} and their exact determination of the remaining residue class when $s=3$ and of the class $n\equiv2\pmod4$ when $s=4$ \cite{FranklKupavskiiSmallMatchings}.
Beyond these results, the general answer remains unknown and may depend heavily on $n\pmod s$.
Hence, throughout the paper, write $n=ms+c=(m+1)s-\ell$, where $0\le c<s$ and $\ell=s-c$.

All known exact values of $e(n,s)$ are attained by weighted constructions.
Let $w\colon [n]\to\mathbb R_{\ge0}$.
If $\sum_{i=1}^n w(i)<s$, then the family
\[
        \mathcal F(w)\coloneqq
        \left\{F\subseteq[n]\colon\sum_{x\in F}w(x)\ge1\right\}
\]
has matching number less than $s$.
Indeed, an $s$-matching $F_1,\ldots,F_s$ in $\mathcal F(w)$ would give
\[
        s>\sum_{i=1}^n w(i)\ge\sum_{r=1}^s\sum_{x\in F_r}w(x)\ge s,
\]
a contradiction.
Frankl and Kupavskii \cite{FKnonuniform} proposed the following meta-conjecture.

\begin{conjecture}[Frankl--Kupavskii \cite{FKnonuniform}]\label{conj:FK-meta}
        For any integers $n\ge s\ge2$, the maximum in the definition of $e(n,s)$ is attained by a family $\mathcal F(w)$ for some $w\colon [n]\to\mathbb R_{\ge0}$ satisfying $\sum_{i=1}^n w(i)<s$.
\end{conjecture}

One classical candidate is
\[
        \mathcal P(m,s,\ell)
        \coloneqq
        \{F\subseteq[n]\colon |F|+|F\cap [\ell-1]|\ge m+1\}.
\]
Indeed, if $w_{\mathcal P}(i)=2/(m+1)$ for $i\le\ell-1$ and $w_{\mathcal P}(i)=1/(m+1)$ for $i\ge\ell$, then $\mathcal P(m,s,\ell)=\mathcal F(w_{\mathcal P})$ and $\sum_i w_{\mathcal P}(i)=s-1/(m+1)<s$.
Frankl and Kupavskii \cite{FKnonuniform} also posed the following conjecture.
\begin{conjecture}[Frankl--Kupavskii \cite{FKnonuniform}]\label{pconj}
        Suppose that $s\ge2$, $m\ge1$ and $n=s(m+1)-\ell$ for some $\ell$ with $1\le\ell\le \lceil s/2\rceil$.
        Then
        \[
                e(n,s)=|\mathcal P(m,s,\ell)|.
        \]
\end{conjecture}
They also verified \cref{pconj} in several ranges.
Kupavskii and Sokolov \cite{KupavskiiSokolovMore} later proved that, for every fixed $m$ and every $\varepsilon>0$, the conjecture holds when $s$ is sufficiently large and $\ell\le(1/2-\varepsilon)s$.
In \cite{ChiWangK1}, the authors of this paper showed that, for every fixed $m\ge3$ and all sufficiently large $s$, $\mathcal P(m,s,\ell)$ is, up to relabelling, the unique extremal family when $1\le\ell\le ((m+1)/(2m+1)-o(1))s$, confirming \cref{pconj}.

The present paper studies another class of weighted constructions.
For $0\le k\le m$, let $a_k=ms-kc-1$.
For $A\in\binom{[n]}{a_k}$, define
\begin{equation}\label{eq:Hfamily-def}
        \mathcal H^k(m,s,\ell;A)
        \coloneqq
        \{F\subseteq[n]:k|F|+|F\cap A|\ge m(k+1)\}.
\end{equation}
We also write $\mathcal H^k(m,s,\ell)\coloneqq\mathcal H^k(m,s,\ell;[a_k])$.
Clearly, $\mathcal H^k(m,s,\ell;A)$ is a relabelling of $\mathcal H^k(m,s,\ell)$.
Define $w_A$ by $w_A(i)=\frac1m$ for $i\in A$ and $w_A(i)=\frac{k}{m(k+1)}$ for $i\notin A$.
By this definition, $\mathcal H^k(m,s,\ell;A)=\mathcal F(w_A)$ and $\sum_i w_A(i)=s-1/(m(k+1))<s$.
Hence, $\nu(\mathcal H^k(m,s,\ell;A))<s$.

The case for small $c$ was studied by Kupavskii and Sokolov \cite{KupavskiiSokolovOtherEnd}.
They showed that for every fixed $m$, among shifted families, $\mathcal H^0(m,s,\ell)$ and $\mathcal H^1(m,s,\ell)$\footnote{In previous papers, $\mathcal H^0(m,s,\ell)$ and $\mathcal H^1(m,s,\ell)$ are often denoted by $\mathcal W(m,s,\ell)$ and $\mathcal Q(m,s,\ell)$, respectively.} are the only extremal families for $e(sm+1,s)$ when $s$ is sufficiently large.  For every fixed $m$ and $c\ge2$, there exists $s_0=s_0(m,c)$ such that $\mathcal H^1(m,s,\ell)$ is the unique shifted extremal family for $e(sm+c,s)$ when $s\ge s_0$.
They also posed the following conjecture:
\begin{equation}\label{equ:KS-conjecture}
        e(n,s)=
        \max\left\{
        |\mathcal H^0(m,s,\ell)|,
        |\mathcal H^1(m,s,\ell)|,
        |\mathcal H^m(m,s,\ell)|,
        |\mathcal P(m,s,\ell)|
        \right\}.
\end{equation}
In \cite{KupavskiiSokolovComplete}, they prove the conjecture when $n \le 3s$, which provides evidence for this conjecture.
Recently, in \cite{ChiWangK2}, the authors of this paper proved that for every $m\ge3$, there are constants $\beta_m,\delta_m>0$ such that $\mathcal H^m(m,s,\ell;A)$\footnote{In previous papers, $\mathcal H^m(m,s,\ell;A)$ is often denoted by $\mathcal P'(m,s,\ell;A)$.} is, up to relabelling, the unique extremal family for $e(sm+c,s)$ when $s$ is sufficiently large and $\beta_ms^{(m-1)/m}\le c\le\delta_ms$.
In the same paper \cite{ChiWangK2}, the authors also showed that $|\mathcal H^{m-1}(m,s,\ell)|$ is larger than the sizes of all four candidates in \eqref{equ:KS-conjecture} for some values of $c$, thus disproving this conjecture of Kupavskii and Sokolov.
Hence, the families $\mathcal H^k(m,s,\ell)$, for $1\le k\le m-1$, are natural extremal candidates for $e(n,s)$.

The main contribution of this paper is to confirm the extremality of these families.

\begin{theorem}\label{thm:main-result-2}
        For every $m\ge3$ and every $1\le k\le m-1$, there exist constants $\alpha=\alpha(m,k)>0$, $\beta=\beta(m,k)>0$ and an integer $s_0=s_0(m,k)$ such that the following holds for all integers $s\ge s_0$ and all integers $c$ with $0\le c<s$.
        Let $n=ms+c$ and $\ell=s-c$.
        If
        \[
                \beta s^{\frac{k-1}{k}}
                \le c\le
                \alpha s^{\frac k{k+1}},
        \]
        then $e(n,s)=|\mathcal H^k(m,s,\ell)|$.
        Moreover, if $\F\subseteq2^{[n]}$ satisfies $\nu(\F)<s$ and $|\F|=e(n,s)$, then $\F=\mathcal H^k(m,s,\ell;A)$ for some $A\in\binom{[n]}{a_k}$, where $a_k=ms-kc-1$.

\end{theorem}

The two exponents in the hypothesis describe the transition scales between consecutive weighted candidates.
The lower scale $s^{(k-1)/k}$ is the upper transition scale associated with $\mathcal H^{k-1}$, while the upper scale $s^{k/(k+1)}$ is the corresponding transition scale from $\mathcal H^k$ toward $\mathcal H^{k+1}$.
Since $(k-1)/k<k/(k+1)$, the interval in \cref{thm:main-result-2} is nonempty for all sufficiently large $s$.
Thus, the theorem places every intermediate family $\mathcal H^k$, $1\le k\le m-1$, in its own polynomial window and supplies infinitely many new exact values of $e(n,s)$.
The uniqueness assertion also identifies the full equality structure in these windows, rather than only the extremal value.

Together with the known results \cite{ChiWangK1,ChiWangK2,KupavskiiSokolovComplete,KupavskiiSokolovOtherEnd,KupavskiiSokolovMore}, \cref{thm:main-result-2} provides further evidence for \cref{conj:FK-meta}, since in all these cases the extremal value is attained by a weighted family.

For a family $\F\subseteq2^{[n]}$, its \emph{upward closure} $\cl(\F)$ is $\{G\subseteq[n]:\text{there exists }F\in\F\text{ with }F\subseteq G\}$.
We say that $\F$ is \emph{upward closed} if $\cl(\F)=\F$.
The proof of \cref{thm:main-result-2} proceeds through a reduction to upward closed families with no members of size less than $m$.
This is the natural setting because $\mathcal H^k(m,s,\ell;A)$ itself contains no set of size below $m$.

We now sketch the proof of \cref{thm:main-result-2}.
We first show that an upward closed family with $\nu(\F)<s$ can have at most $|\mathcal H^k(m,s,\ell)|$ elements (see \cref{thm:reduced-all}).
Let $\F_m$ be the $m$-th layer of $\F$.
We choose an $(ms-kc-1)$-set $A$ that maximizes $g(A)=|\F_m\cap\binom Am|$.
We then compare the surplus $\F\setminus\mathcal H^k(m,s,\ell;A)$ with the deficit $\mathcal H^k(m,s,\ell;A)\setminus\F$.
An analysis of matchings among the surplus sets shows that the deficit exceeds the surplus unless the surplus is concentrated around a small collection of vertices of $A$.
In this exceptional configuration, replacing those vertices by suitable vertices of $R=[n]\setminus A$ produces another set $A'$ with $g(A')>g(A)$, contradicting the choice of $A$.
Hence, the exceptional configuration cannot occur, and the deficit exceeds the surplus unless both are empty. Consequently, $\F=\mathcal H^k(m,s,\ell;A)$ in the equality case.

Now we reduce the problem to upward closed families with no member of size less than $m$.
Starting from $\F$ with $\nu(\F)<s$, we replace $\emptyset$, if necessary, by a suitable singleton and then take the upward closure.
These operations do not decrease the size of the family or increase its matching number.
If the resulting family has no member of size less than $m$, then \cref{thm:reduced-all} applies.
Otherwise, fix a $q$-set $B$ with $q<m$.
Reserving $B$ as one member of a potential matching reduces the problem to the ground set $[n]\setminus B$ and parameters $s'=s-1$ and $c'=c+m-q$.
\cref{lem:main-low-layer-gap} shows that the number of missing sets of the corresponding shifted candidate, even after all layers below $m$ are included, exceeds that of $\mathcal H^k(m,s,\ell)$ by a term of order $s^{m-1}$.
Consequently, if the upward closure contains a set of size less than $m$, then $|\F|<|\mathcal H^k(m,s,\ell)|$.

Theorem~\ref{thm:main-result-2} suggests that, after including all the families $\mathcal H^k$, the known weighted constructions may form a complete list of extremal families.
We therefore propose the following conjecture.

\begin{conjecture}\label{conj:complete-classification}
        For every $n$ and $s$ with $n=sm+c=s(m+1)-\ell$, we have
        \[
                e(n,s)=\max\left\{|\mathcal H^0(m,s,\ell)|, |\mathcal H^1(m,s,\ell)|,\ldots,|\mathcal H^m(m,s,\ell)|,|\mathcal P(m,s,\ell)|\right\}.
        \]
        Moreover, every extremal family is isomorphic to one of the
        candidates attaining this maximum.
\end{conjecture}

This paper is organized as follows.
In \cref{sec:prelim}, we collect the auxiliary notation and results of matchings used throughout the proof.
In \cref{sec:proof-main}, we reduce the problem to upward-closed families and prove \cref{thm:main-result-2}, assuming \cref{thm:reduced-all}.
In \cref{sec:proof-reduced}, we show the extremality of upward closed families and thus complete the proof of \cref{thm:reduced-all}.

\section{Preliminaries}\label{sec:prelim}

Given a family $\mathcal A\subseteq 2^{[n]}$ and an integer $r\ge0$, write $\mathcal A_r$ for the $r$-layer of $\mathcal A$, namely, $\mathcal A_r=\mathcal A\cap \binom{[n]}{r}$.
In particular, let $\mathcal H^k_r(m,s,\ell;A)$ denote the
$r$-layer of $\mathcal H^k(m,s,\ell;A)$, that is,
\begin{equation}\label{eq:actual-layer-def}
        \mathcal H^k_r(m,s,\ell;A)=
        \left\{F\in\binom{[n]}r:kr+|F\cap A|\ge m(k+1) \right\}.
\end{equation}
For $1\le k\le m-1$, let $L\coloneqq\left\lfloor\frac{m-1}{k}\right\rfloor$.
For $0\le j\le L$, let $\tau_j\coloneqq m-kj$. Then
\[
        \mathcal H^k_{m+j}(m,s,\ell;A)
        =
        \left\{F\in\binom{[n]}{m+j}\colon |F\cap A|\ge\tau_j\right\}.
\]
Moreover, for $i\ge m+L+1$, the $i$-layer of
$\mathcal H^k(m,s,\ell;A)$ is full, that is,
$\mathcal H^k_i(m,s,\ell;A)=\binom{[n]}i$.

Given a finite set $V$ and a family $\mathcal A\subseteq2^V$, write $d_{\mathcal A}(i)$ for the {\it degree} of $i$, namely, the number of sets that contain $i$ in $\mathcal A$.
Let $\delta(\mathcal A)=\min\{d_{\mathcal A}(i)\colon i\in V\}$ be the \emph{minimum degree} of $\mathcal A$.
Also, let $\Delta(\mathcal A)=\max\{d_{\mathcal A}(i)\colon i\in V\}$ denote the \emph{maximum degree} of $\mathcal A$.
Given a family $\mathcal A\subseteq2^V$, a matching in $\mathcal A$ is called a \emph{perfect matching} if its union is $V$.
For a family $\mathcal A\subseteq2^V$ and $W\subseteq V$, write $\mathcal A[W]=\{E\in\mathcal A:E\subseteq W\}$ for the subfamily induced by $W$.

Define
\begin{equation}\label{eq:defect-proof-E-B}
        E_k(s,c)=s^{m-2}+cs^{m-2}+
        \sum_{j=1}^{L}\sum_{x=0}^{\tau_j-1}s^x c^{m+j-x-1} \text{ and }
        B_k(s,c)=s^{m-1}+E_k(s,c).
\end{equation}
Throughout the proof, we will use $E_k(s,c)$ and $B_k(s,c)$ to control lower-order error terms.

We shall use the following standard form of Vandermonde's identity; see, for example, \cite{StanleyEC1}.

\begin{lemma}\label{lem:vandermonde}
        Let $x,y,t$ be nonnegative integers.
        We have $\sum_{b=0}^{t}\binom xb\binom y{t-b} = \binom{x+y}{t}$.
        In particular, for every integer $m\ge1$, if $x+y=N$, then $\sum_{b=0}^{m-1}\binom xb\binom y{m-b} = \binom Nm-\binom xm$.
\end{lemma}

We next use a minimum-degree condition that guarantees a perfect matching in a uniform family.

\begin{lemma}[Daykin and H\"aggkvist \cite{DaykinHaggkvist}]
        \label{lem:daykin-haggkvist-input}
        For every integer $m\ge2$, there exists an integer
        $C_{\ref{lem:daykin-haggkvist-input}}
                =C_{\ref{lem:daykin-haggkvist-input}}(m)$ such that the
        following holds.  Let $V$ be a finite set with
        $N=|V|\ge C_{\ref{lem:daykin-haggkvist-input}}$ and $m\mid N$.
        If $G\subseteq\binom Vm$ satisfies $\delta(G)>\left(1-\frac1m\right)\binom{N-1}{m-1}$, then $G$ contains a perfect matching.
\end{lemma}

Using \cref{lem:daykin-haggkvist-input}, we prove the following lemma.

\begin{lemma}\label{lem:dense-defect}
        For every integer $m\ge2$ and every $\theta>0$, there exist
        constants $\zeta=\zeta(m,\theta)>0$ and
        $v_{\ref{lem:dense-defect}}=v_{\ref{lem:dense-defect}}(m,\theta)$
        such that the following holds.  Let $V$ be a finite set with
        $N=|V|\ge v_{\ref{lem:dense-defect}}$ and $m\mid N$.
        If $G\subseteq\binom Vm$ satisfies $\delta(G)\ge \theta N^{m-1}$, then either $G$ has a perfect matching or $\left|\binom Vm\setminus G\right|\ge \zeta N^m$.
\end{lemma}

\begin{proof}
        Let $\beta_1=1/8$.
        Choose
        \[
                0<\beta_0<\min\left\{\frac{\beta_1}{2m},
                \frac{\theta}{10m(2m)^{m-2}}\right\},
        \]
        and then choose
        \[
                0<\varepsilon<\frac{1}{2m(m-1)!}(1-\beta_1)^{m-1}.
        \]
        Choose $\zeta>0$ so that $m\zeta/\varepsilon\le\beta_0$, and choose $v_{\ref{lem:dense-defect}}$ sufficiently large so that \cref{lem:daykin-haggkvist-input} applies whenever $N'\ge (1-\beta_1)N$, $N\ge v_{\ref{lem:dense-defect}}$, and $m\mid N'$, and so that
        \begin{equation}\label{eq:DH-delta-explicit}
                \varepsilon N^{m-1}<\frac1m\binom{N'-1}{m-1}
        \end{equation}
        whenever $N\ge v_{\ref{lem:dense-defect}}$ and $N'\ge(1-\beta_1)N$.

        Now let $V$ and $G\subseteq\binom Vm$ satisfy the assumptions of the lemma.
        If $\left|\binom Vm\setminus G\right|\ge\zeta |V|^m$, then the proof is complete.
        Thus, suppose that $\left|\binom Vm\setminus G\right|<\zeta |V|^m$.
        We shall show that $G$ has a perfect matching.

        Set $Z=\binom Vm\setminus G$, and let $B=\{x\in V:d_Z(x)>\varepsilon |V|^{m-1}\}$.
        Since $\sum_x d_Z(x)=m|Z|$, we have $|B|<m\zeta |V|/\varepsilon\le\beta_0|V|$.

        We construct $\mathcal C$ inductively, starting with $\mathcal C=\emptyset$, and maintaining that $\mathcal C$ is a matching in $G$.
        Suppose that $\mathcal C$ has been constructed and does not yet cover $B$.
        Choose $x\in B\setminus\bigcup_{E\in\mathcal C}E$, and let $W=\bigcup_{E\in\mathcal C}E$.
        We shall find a set $E\in G$ with $x\in E$ and $E\cap W=\emptyset$.
        Adding such a set preserves the matching property.

        Since each previously chosen set covered at least one new vertex of $B$, before this step fewer than $|B|$ sets have been chosen.
        Hence, $|W|=m|\mathcal C|<m|B|\le m\beta_0|V|$.
        A member of $G$ containing $x$ cannot be added to the current matching only if it meets $W$.
        For each $y\in W$, the number of $m$-sets containing both $x$ and $y$ is at most $\binom{|V|}{m-2}$.
        Therefore, the number of members of $G$ containing $x$ that cannot be added is at most
        \[
                |W|\binom{|V|}{m-2}
                \le m\beta_0|V|\binom{|V|}{m-2}
                <\frac{\theta}{2}|V|^{m-1},
        \]
        where the last inequality follows from the choices of $\beta_0$ and $v_{\ref{lem:dense-defect}}$.
        Since $d_G(x)\ge\theta |V|^{m-1}$, there exists a member of $G$ containing $x$ and disjoint from $W$.
        We add this set to $\mathcal C$.
        Continuing in this way, we obtain a matching $\mathcal C\subseteq G$ whose union contains $B$.

        Let $V'=V\setminus\bigcup_{E\in\mathcal C}E$.
        Since the sets in $\mathcal C$ are pairwise disjoint $m$-sets, we have $|V'|=|V|-m|\mathcal C|$.
        In particular, as $m\mid |V|$, the integer $|V'|$ is still divisible by $m$.
        Also, since $|\mathcal C|\le |B|$,
        \[
                |V'|\ge(1-m\beta_0)|V|\ge(1-\beta_1)|V|.
        \]

        Let $Z'=\binom{V'}m\setminus G[V']$.
        Since $B\subseteq\bigcup_{E\in\mathcal C}E$, every $x\in V'$ lies outside $B$.
        Hence, $d_Z(x)\le \varepsilon |V|^{m-1}$.
        Moreover, every missing $m$-set of $G[V']$ containing $x$ is also a missing $m$-set of $G$ containing $x$.
        Thus, $d_{Z'}(x)\le d_Z(x)\le \varepsilon |V|^{m-1}$.
        Consequently, for every $x\in V'$, $d_{G[V']}(x) \ge \binom{|V'|-1}{m-1}-\varepsilon |V|^{m-1}$.
        By \eqref{eq:DH-delta-explicit}, we obtain
        \[
                \delta(G[V'])
                \ge \binom{|V'|-1}{m-1}-\varepsilon |V|^{m-1}
                > \left(1-\frac1m\right)\binom{|V'|-1}{m-1}.
        \]
        Therefore, \cref{lem:daykin-haggkvist-input} gives a perfect matching in $G[V']$.
        Together with the sets in $\mathcal C$, this gives a perfect matching of $G$.
        Thus, whenever $\left|\binom Vm\setminus G\right|<\zeta |V|^m$, $G$ has a perfect matching.
        This proves the lemma.
\end{proof}

Given a set $A$ and an integer $q$, write $A=A_1\sqcup\cdots\sqcup A_q$ if $A$ is the disjoint union of $A_1,\ldots,A_q$.
Given integers $p,b,u,h$ with $p,b,u\ge1$ and $0\le h\le u$, let
$P$ and $X$ be disjoint sets with $|P|=pu$ and $|X|=bu+h$.
For each $e\in\binom Pp$, let $\G_e^{(b)}\subseteq\binom Xb$ and $\G_e^{(b+1)}\subseteq\binom X{b+1}$.
A \emph{mixed ordered partition} is a pair of ordered partitions $P=e_1\sqcup\cdots\sqcup e_u$ and $X=T_1\sqcup\cdots\sqcup T_u$ such that $|e_r|=p$ for every $r$, exactly $h$ of the $T_r$'s have size $b+1$, the remaining $u-h$ have size $b$, and $T_r\in\G_{e_r}^{(|T_r|)}$ for every $r\in[u]$.

The following lemma on mixed ordered partitions will be used in the proof of \cref{thm:reduced-all}.

\begin{lemma}[Chi and Wang \cite{ChiWangK2}]\label{lem:p-ordered-local-blocker}
        Let $m,p$ be integers with $m\ge2$ and $1\le p\le m-1$, and set $b=m+1-p$.
        There exist constants $\gamma_{\ref{lem:p-ordered-local-blocker}}>0$, $\rho_{\ref{lem:p-ordered-local-blocker}}>0$, and $C_{\ref{lem:p-ordered-local-blocker}}$ depending only on $m,p$ such that the following holds.
        Let $u,h$ be integers with $u\ge C_{\ref{lem:p-ordered-local-blocker}}$ and $0\le h\le\rho_{\ref{lem:p-ordered-local-blocker}}u$.
        Let $P,X$ and the families $\G_e^{(b)},\G_e^{(b+1)}$ be as in the definition above.
        If no mixed ordered partition exists, then
        \[
                \sum_{e\in\binom Pp}
                \left(
                \left|\binom Xb\setminus\G_e^{(b)}\right|
                +
                \left|\binom X{b+1}\setminus\G_e^{(b+1)}\right|
                \right)
                \ge \gamma_{\ref{lem:p-ordered-local-blocker}}(h+1)u^m.
        \]
\end{lemma}

\medskip
\noindent {\bf Remark.} Later, in the proof of \cref{thm:reduced-all}, we may and shall assume that
$0<\rho_{\ref{lem:p-ordered-local-blocker}}\le1$ and
$C_{\ref{lem:p-ordered-local-blocker}}\ge1$.
Indeed, replacing $\rho_{\ref{lem:p-ordered-local-blocker}}$ by
$\min\{\rho_{\ref{lem:p-ordered-local-blocker}},1\}$ and replacing
$C_{\ref{lem:p-ordered-local-blocker}}$ by
$\max\{C_{\ref{lem:p-ordered-local-blocker}},1\}$ only strengthens
the hypotheses and leaves the conclusion unchanged.

\section{Reduction to upward closed families}\label{sec:proof-main}

We first state the following extremal result for upward closed families with no members of size less than $m$.
We will prove this result in \cref{sec:proof-reduced}.
\begin{theorem}\label{thm:reduced-all}
        Fix $m\ge3$ and $1\le k\le m-1$.
        There exist constants $\alpha'=\alpha'(m,k)\in(0,1)$ and $\beta'=\beta'(m,k)>1$, and an integer $s_0'=s_0'(m,k)$, such that the following holds for all integers $s\ge s_0'$ and all integers $c$ with $0\le c<s$.
        Let $n=ms+c$ and set $\ell=s-c$.
        Suppose that $\beta' s^{\frac{k-1}{k}}\le c\le\alpha' s^{\frac k{k+1}}$.
        Let $\F\subseteq2^{[n]}$ be an upward closed family with $\nu(\F)<s$ and $\F\cap\binom{[n]}{<m}=\emptyset$.
        We have $|\F|\le |\mathcal H^k(m,s,\ell)|$, with equality if and only if $\F=\mathcal H^k(m,s,\ell;A)$ for some $A\in\binom{[n]}{a_k}$, where $a_k=ms-kc-1$.
\end{theorem}

Assuming \cref{thm:reduced-all}, we now prove \cref{thm:main-result-2} by reducing an arbitrary family to this setting.
First we remove a possible empty set from $\F$ without changing its size, and then take the upward closure of $\F$.
If the resulting upward closed family contains no nonempty set of size less than $m$, then \cref{thm:reduced-all} applies directly.
Otherwise it contains a $q$-set $B$, where $1\le q\le m-1$.
We reserve $B$ as one member of a potential matching, pass to the family induced by $[n]\setminus B$, and apply \cref{thm:reduced-all} with $s'=s-1$ and $c'=c+m-q$.

\subsection{The low-layer gap}

Let $\alpha',\beta',s_0'$ be the constants from \cref{thm:reduced-all}.
For fixed $k,s,c$, write $n(s,c)=ms+c$, $a(s,c)=ms-kc-1$ and $r(s,c)=n(s,c)-a(s,c)=(k+1)c+1$.
Thus, $n(s,c)=ms+c$ is the size of the ground set of $\mathcal H^k(m,s,\ell)$.
Moreover, $a(s,c)$ is the size of the center of $\mathcal H^k(m,s,\ell)$ and $r(s,c)$ is the size of its complement.
For comparison later, let $D_k(s,c)$ be the number of sets missing from $\mathcal H^k(m,s,\ell)$, namely,
\[
        D_k(s,c)\coloneqq 2^{n(s,c)}-|\mathcal H^k(m,s,\ell)|.
\]
Since $\mathcal H^k(m,s,\ell;A)$ contains no set below size $m$,
contains every layer of size at least $m+L+1$, and in layer $m+j$
for $0\le j\le L$ misses exactly the sets with fewer than $\tau_j$
elements in the center $A$, we have
\[
        D_k(s,c)=
        \sum_{i=0}^{m-1}\binom{n(s,c)}{i}
        +
        \sum_{j=0}^L\sum_{b=0}^{\tau_j-1}
        \binom{a(s,c)}{b}\binom{r(s,c)}{m+j-b}.
\]

For the shifted comparison, it is convenient to remove the low layers from $D_k(s,c)$.
Thus, define
\[
        \overline{D}_k(s,c)=
        D_k(s,c)-\sum_{i=0}^{m-1}\binom{n(s,c)}{i}.
\]
The next lemma gives a uniform lower bound on
$\overline{D}_k(s-1,c+m-q)-D_k(s,c)$ for $1\le q\le m-1$.
This estimate is the quantitative ingredient needed to complete the reduction in the next subsection.

\begin{lemma}\label{lem:main-low-layer-gap}
        There exist constants $C_k=C_k(m,k)>0$, $P_k=P_k(m,k)>0$ and an integer $s_1=s_1(m,k)$ such that $C_k\ge \beta'$, $P_k\le \alpha'$, $s_1\ge s_0'$, and the following holds.
        Suppose $s\ge s_1$ and $C_k s^{(k-1)/k}\le c\le P_k s^{k/(k+1)}$.
        Let $n=ms+c$.
        For every $1\le q\le m-1$, let $s'=s-1$, $c'=c+m-q$, and $n'=n-q=ms'+c'$.
        We have $s'\ge s_0'$, $\beta'(s')^{(k-1)/k}\le c'\le \alpha'(s')^{k/(k+1)}$ and $c'<s'$.
        Moreover,
        \[
                \overline{D}_k(s',c')
                \ge
                D_k(s,c)+\frac{k}{2}\frac{m^{m-1}}{(m-1)!}s^{m-1}.
        \]
\end{lemma}

In the next subsection, we prove \cref{thm:main-result-2} assuming \cref{lem:main-low-layer-gap}.
The proof of \cref{lem:main-low-layer-gap} is given in \cref{subsec:proofofclaim}.

\subsection{Closing the reduction}\label{subsec:close}

\begin{proof}[Proof of \cref{thm:main-result-2} from \cref{lem:main-low-layer-gap}]
        Fix $m\ge3$ and $1\le k\le m-1$, and take $C_k,P_k,s_1$ from \cref{lem:main-low-layer-gap}.
        Set $\beta=C_k$, $\alpha=P_k$, and $s_0=s_1$.
        Let $s\ge s_0$ and let $c$ satisfy
        \[
                \beta s^{(k-1)/k}\le c\le \alpha s^{k/(k+1)}.
        \]
        Set $n=ms+c$ and $\ell=s-c$.

        For any $A\in\binom{[n]}{a_k}$, as observed in \cref{sec:intro},
        the family $\mathcal H^k(m,s,\ell;A)$ has matching number less than $s$.
        This gives the lower bound for $e(n,s)$.

        Now we prove the upper bound.
        Let $\F\subseteq2^{[n]}$ satisfy $\nu(\F)<s$.
        We first remove the empty set if it exists.
        If $\emptyset\in\F$, then $\F$ cannot contain every singleton, because any $s$ distinct singletons would form an $s$-matching.
        Choose $x\in[n]$ with $\{x\}\notin\F$, and let
        \[
                \F'=(\F\setminus\{\emptyset\})\cup\{\{x\}\}.
        \]
        The family $\F'$ satisfies $|\F'|=|\F|$, $\emptyset\notin\F'$, and $\nu(\F')<s$.
        Indeed, an $s$-matching in $\F'$ that avoids $\{x\}$ is already an $s$-matching in $\F$, while an $s$-matching in $\F'$ that uses $\{x\}$ leaves an $(s-1)$-matching in $\F\setminus\{\emptyset\}$; adjoining $\emptyset$ would then give an $s$-matching in $\F$.
        If $\emptyset\notin\F$, set $\F'=\F$.

        Since $\emptyset\notin\F'$, taking upward closure preserves the matching-number bound.
        Indeed, if $G_1,\ldots,G_t\in\cl(\F')$ are pairwise disjoint, then for each $i$ there is a nonempty $G_i'\in\F'$ with $G_i'\subseteq G_i$; the sets $G_1',\ldots,G_t'$ are pairwise disjoint.
        Hence, $\nu(\cl(\F'))<s$, and $|\F|=|\F'|\le |\cl(\F')|$.

        Suppose first that $\cl(\F')\cap\binom{[n]}{<m}=\emptyset$.
        Since $C_k\ge \beta'$, $P_k\le \alpha'$, and $s\ge s_0=s_1\ge s_0'$, the pair $(s,c)$ lies in the admissible range of \cref{thm:reduced-all}.
        Therefore, $\cl(\F')$ satisfies the hypotheses of \cref{thm:reduced-all}, and $|\F|\le |\cl(\F')|\le |\mathcal H^k(m,s,\ell)|$.
        This proves the upper bound in this case.

        We may therefore assume that $\cl(\F')$ contains a $q$-set $B$ with $1\le q\le m-1$.
        Set $V=[n]\setminus B$ and $\F''=\{F\in\cl(\F'):F\subseteq V\}$.
        Also set $s'=s-1$, $c'=c+m-q$, and $n'=n-q=ms'+c'$.
        By \cref{lem:main-low-layer-gap}, the pair $(s',c')$ satisfies the hypotheses of \cref{thm:reduced-all}.

        The family $\F''$ has no $(s-1)$-matching.
        Otherwise, such a matching together with $B\in\cl(\F')$ would give an $s$-matching in $\cl(\F')$.
        Let
        \[
                \F''_{\ge m}\coloneqq
                \F''\cap\binom V{\ge m}.
        \]
        Then $\F''_{\ge m}$ is upward closed in $V$, contains no set below size $m$, and has matching number less than $s'$.
        Applying \cref{thm:reduced-all} on $V$, after relabelling $V$ as $[n']$, gives
        \[
                |\F''_{\ge m}|\le 2^{n'}-D_k(s',c').
        \]
        Consequently,
        \[
                |\F''|=
                \left|\F''\cap\binom V{<m}\right|
                +|\F''_{\ge m}| \le
                \sum_{i=0}^{m-1}\binom{n'}i
                +2^{n'}-D_k(s',c')  =2^{n'}-\overline D_k(s',c').
        \]
        Every member of $\cl(\F')$ either intersects $B$ or belongs to $\F''$.
        Thus, using the bound on $\F''$ and \cref{lem:main-low-layer-gap}, we obtain
        \[
                \begin{aligned}
                        |\F|
                         & \le |\cl(\F')|
                        \le (2^n-2^{n'})+|\F''|
                        \le 2^n-\overline{D}_k(s',c')  \le 2^n-D_k(s,c)
                        -\frac{k}{2}\frac{m^{m-1}}{(m-1)!}s^{m-1} \\
                         & =|\mathcal H^k(m,s,\ell)|
                        -\frac{k}{2}\frac{m^{m-1}}{(m-1)!}s^{m-1}  <|\mathcal H^k(m,s,\ell)|.
                \end{aligned}
        \]
        This proves the upper bound in all cases.

        Finally we consider equality.
        Suppose $|\F|=|\mathcal H^k(m,s,\ell)|$.
        The strict low-layer case above is impossible.
        In particular, $\emptyset\notin\F$: if $\emptyset\in\F$, then the replacement step places a singleton into $\F'$, hence into $\cl(\F')$, and the strict low-layer case applies.
        Thus, $\F'=\F$.

        The upward closure $\cl(\F)$ contains no nonempty set below size $m$.
        If $\cl(\F)\ne\F$, then
        \[
                |\cl(\F)|>|\F|=|\mathcal H^k(m,s,\ell)|,
        \]
        contradicting \cref{thm:reduced-all}, since $\cl(\F)$ is upward closed, has matching number less than $s$, and contains no set below size $m$.
        Hence, $\F$ is upward closed and contains no set below size $m$.
        The equality statement in \cref{thm:reduced-all} gives $\F=\mathcal H^k(m,s,\ell;A)$ for some $A\in\binom{[n]}{a_k}$, where $a_k=ms-kc-1$.
\end{proof}

\subsection{Proof of Lemma~\ref{lem:main-low-layer-gap}}\label{subsec:proofofclaim}

Fix $q\in[m-1]$, and use $s'$, $c'$, and $n'$ as in the statement.
Let $a=ms-kc-1$, $r=(k+1)c+1$, $a'=ms'-kc'-1$ and $r'=(k+1)c'+1$.
Thus, $a+r=n$ and $a'+r'=n'$.
Moreover, $a'=a-((k+1)m-kq)$ and $r'=r+(k+1)(m-q)$.

Here $a$ is the size of the center and $r$ is the complement size for the original parameters $(s,c)$.
The quantities $a'$ and $r'$ are the sizes of the corresponding center and complement for the new parameters $(s',c')$, respectively.

The goal is to prove that $\overline D_k(s',c')-D_k(s,c)$ is bounded below by a positive constant times $s^{m-1}$.
To this end, we decompose the difference $\overline D_k(s',c')-D_k(s,c)$ into three terms:
\[
        \overline D_k(s',c')-D_k(s,c)=X+Y+Z,
\]
where
\begin{align*}
        X
         & =
        \left(
        \sum_{b=0}^{m-1}
        \binom{a'}b\binom{r'}{m-b}
        -
        \sum_{b=0}^{m-1}
        \binom ab\binom r{m-b}
        \right)
        -
        \binom n{m-1}, \\
        Y
         & =
        \sum_{j=1}^{L}\sum_{b=0}^{\tau_j-1}
        \left(
        \binom{a'}b\binom{r'}{m+j-b}
        -
        \binom ab\binom r{m+j-b}
        \right) \text{ and }Z=
        -
        \sum_{i=0}^{m-2}\binom ni.
\end{align*}
We now show that $X$ is the principal term, namely the term of order $s^{m-1}$, while both $Y$ and $Z$ are of order $O(E_k(s,c))$.

We first estimate $Z$. Since $Z=-\sum_{i=0}^{m-2}\binom ni$ and $n=O_m(s)$, we have $Z=O_m(s^{m-2})$, and hence $Z=O(E_k(s,c))$.

Now we estimate $X$.
By \cref{lem:vandermonde}, applied with $x=a$, $y=r$, and $a+r=n$, $\sum_{b=0}^{m-1} \binom ab\binom r{m-b} = \binom nm-\binom am$.
Similarly, we obtain $\sum_{b=0}^{m-1} \binom{a'}b\binom{r'}{m-b} = \binom{n'}m-\binom{a'}m$, and therefore
\[
        X
        =
        \binom{n'}m-\binom nm
        +
        \binom am-\binom{a'}m
        -
        \binom n{m-1}.
\]

Since $n=ms+c$ and $a=ms-kc-1$, we have $n=ms+O(c)$ and $a=ms+O(c)$.
Also $n'-n=-q$ and $a-a'=((k+1)m-kq)$, both bounded in terms of $m,k$.
Since $n'=n-q$ and $a=a'+((k+1)m-kq)$, we obtain
\begin{align*}
        \binom{n'}m-\binom nm
         & =
        -q\frac{(ms)^{m-1}}{(m-1)!}
        +O_{m,k}(cs^{m-2}), \\
        \binom am-\binom{a'}m
         & =
        ((k+1)m-kq)\frac{(ms)^{m-1}}{(m-1)!}
        +O_{m,k}(cs^{m-2}).
\end{align*}
Moreover, $\binom n{m-1} = \frac{(ms)^{m-1}}{(m-1)!}+O_{m,k}(cs^{m-2})$.
Therefore,
\[
        X
        =
        \left((k+1)(m-q)-1\right)
        \frac{(ms)^{m-1}}{(m-1)!}
        +O_{m,k}(cs^{m-2}).
\]

It remains to bound $Y$, the contribution from layers above $m$.
Fix $1\le j\le L$ and $0\le b\le\tau_j-1$.
We estimate the corresponding summand directly.
We have
\begin{align*}
         & \left|
        \binom{a'}b\binom{r'}{m+j-b}
        -
        \binom ab\binom r{m+j-b}
        \right|      \\
         & \qquad\le
        \left(\binom ab-\binom{a'}b\right)
        \binom{r'}{m+j-b}
        +
        \binom ab
        \left(
        \binom{r'}{m+j-b}-\binom r{m+j-b}
        \right).
\end{align*}
Since $a-a'$ is bounded in terms of $m,k$, for $b\ge1$ we have $\binom ab-\binom{a'}b = \sum_{t=0}^{a-a'-1}\binom{a-t-1}{b-1} = O_{m,k}(s^{b-1})$, while this difference is zero when $b=0$.
Also $r'=O_{m,k}(c)$, and hence $\left(\binom ab-\binom{a'}b\right) \binom{r'}{m+j-b} = O_{m,k}\left(s^{b-1}c^{m+j-b}\right)$, with the term omitted when $b=0$.

Similarly, since $r'-r$ is bounded in terms of $m,k$,
\[
        \binom{r'}{m+j-b}-\binom r{m+j-b}
        =
        \sum_{t=0}^{r'-r-1}\binom{r+t}{m+j-b-1}
        =
        O_{m,k}\left(c^{m+j-b-1}\right).
\]
Since $\binom ab=O_{m,k}(s^b)$, this gives $\binom ab \left( \binom{r'}{m+j-b}-\binom r{m+j-b} \right) = O_{m,k}\left(s^bc^{m+j-b-1}\right)$.
Therefore,
\[
        \binom{a'}b\binom{r'}{m+j-b}
        -
        \binom ab\binom r{m+j-b}
        =
        O_{m,k}
        \left(
        s^{b-1}c^{m+j-b}
        +
        s^bc^{m+j-b-1}
        \right),
\]
with the first term omitted when $b=0$.
Summing over all $1\le j\le L$ and $0\le b\le\tau_j-1$, we obtain
\[
        Y
        =O_{m,k}\left(\sum_{j=1}^{L}\sum_{b=0}^{\tau_j-1}
        s^{b-1}c^{m+j-b}
        +
        s^bc^{m+j-b-1}
        \right)=
        O_{m,k}\left(
        \sum_{j=1}^{L}\sum_{b=0}^{\tau_j-1}
        s^bc^{m+j-b-1}
        \right),
\]
where the last equality follows from $c\le s$.

Combining the estimates for $X$, $Y$, and $Z$, there exists a constant $K=K(m,k)$ such that, whenever $1\le c\le s/2$,
\begin{equation}\label{eq:low-layer-clean-estimate}
        \left|
        \overline D_k(s',c')-D_k(s,c)
        -
        \left((k+1)(m-q)-1\right)
        \frac{(ms)^{m-1}}{(m-1)!}
        \right|
        \le K E_k(s,c).
\end{equation}

We now choose the constants in the lemma.
Let $C_k=2\beta'$.
Choose $0<P_k\le \min\{\alpha'/4,1\}$ sufficiently small so that
\[
        K
        \sum_{j=1}^{L}\sum_{b=0}^{\tau_j-1}
        P_k^{(k+1)j}
        \le
        \frac14 k\frac{m^{m-1}}{(m-1)!}.
\]
This is possible because $L$ and the $\tau_j$'s are fixed in terms of $m,k$, and the finite sum on the left tends to $0$ as $P_k$ tends to $0^+$.

Recall that $E_k(s,c)=s^{m-2}+cs^{m-2}+ \sum_{j=1}^{L}\sum_{x=0}^{\tau_j-1}s^x c^{m+j-x-1}$.
Suppose $c\le P_k s^{k/(k+1)}$.
For every summand $s^bc^{m+j-b-1}$ in the double sum defining
$E_k(s,c)$, we have $b\le\tau_j-1=m-kj-1$.
Hence, $m+j-b-1\ge (k+1)j$ and $b+\frac{k}{k+1}(m+j-b-1)\le m-1$.
Therefore,
\[
        s^bc^{m+j-b-1}
        \le
        P_k^{m+j-b-1}
        s^{b+\frac{k}{k+1}(m+j-b-1)}
        \le
        P_k^{(k+1)j}s^{m-1}.
\]
Summing over all admissible pairs and using the choice of $P_k$, the
contribution of the double sum to $K E_k(s,c)$ is at most
$\frac14 k\frac{m^{m-1}}{(m-1)!}s^{m-1}$.
After $P_k$ is fixed, choose an integer $s_1\ge s_0'$ such that,
for all $s\ge s_1$, all $1\le q\le m-1$, and all
$c\le P_k s^{k/(k+1)}$, the following hold:
\[
        c\le s/2, \text{ }
        K(s^{m-2}+cs^{m-2})
        \le
        \frac14 k\frac{(ms)^{m-1}}{(m-1)!}, \text{ }
        P_k s^{k/(k+1)}+m
        \le \alpha'(s-1)^{k/(k+1)}, \text{ }
        s-1\ge s_0'
\]
and $P_k s^{k/(k+1)}+m<s-1$.
Consequently, $K E_k(s,c) \le \frac12 k\frac{(ms)^{m-1}}{(m-1)!}$.

By the choice of $s_1$, the new parameters $(s',c')$ are admissible.
Indeed, if $C_k s^{(k-1)/k}\le c\le P_k s^{k/(k+1)}$, then
\[
        c'=c+m-q\ge c\ge 2\beta' s^{(k-1)/k}
        \ge \beta'(s')^{(k-1)/k},
\]
while $P_k\le\alpha'/4$ gives, by the choice of $s_1$,
\[
        c'=c+m-q
        \le P_k s^{k/(k+1)}+m
        \le \alpha'(s')^{k/(k+1)}.
\]
Also $s'\ge s_0'$ and $c'<s'$ by the choice of $s_1$.
Finally, for $1\le q\le m-1$, $(k+1)(m-q)-1\ge k$.
Thus, \eqref{eq:low-layer-clean-estimate} gives
\[
        \overline D_k(s',c')-D_k(s,c)
        \ge
        k\frac{(ms)^{m-1}}{(m-1)!}
        -
        \frac12 k\frac{(ms)^{m-1}}{(m-1)!}=\frac12 k\frac{m^{m-1}}{(m-1)!}s^{m-1}.
\]
Hence, the lemma follows.

\section{Extremality of upward closed families}\label{sec:proof-reduced}

In this section, we prove \cref{thm:reduced-all}.
Let $\F_m$ be the $m$-th layer of $\F$.
We choose an $(ms-kc-1)$-set $A$ that maximizes $g(A)=|\F_m\cap\binom{A}{m}|$.
For this choice of $A$, we compare the surplus $\F\setminus\mathcal H^k(m,s,\ell;A)$ with the deficit $\mathcal H^k(m,s,\ell;A)\setminus\F$.
Since
\[
        |\F|-|\mathcal H^k(m,s,\ell;A)|
        =
        |\F\setminus\mathcal H^k(m,s,\ell;A)|
        -
        |\mathcal H^k(m,s,\ell;A)\setminus\F|,
\]
it suffices to show that the surplus is smaller than the deficit unless
both set differences are empty. We establish this comparison using
matching estimates together with the maximality of $g(A)$. This also
shows that equality holds precisely when
$\F=\mathcal H^k(m,s,\ell;A)$.

\subsection{Comparison with a \texorpdfstring{$g$}{g}-maximizing center}

Throughout this section we fix integers $m\ge3$ and $1\le k\le m-1$.
Let $b=k+1$, $p=m-k$, $L=\left\lfloor\frac{m-1}{k}\right\rfloor$, and $\tau_j=m-kj$ for $0\le j\le L$.
Thus, $p+b=m+1$ and $\tau_1=p$.
All implicit constants below depend only on $m$ and $k$.

The constants $s_0'$, $\alpha'$, and $\beta'$ in \cref{thm:reduced-all} will be chosen below.
Throughout the proof, let $s\ge s_0'$ and let $c$ satisfy
\[
        \beta' s^{\frac{k-1}{k}}\le c\le \alpha' s^{\frac k{k+1}}.
\]
Set $n=ms+c$ and $\ell=s-c$.
Let $\F\subseteq2^{[n]}$ be upward closed, satisfy $\nu(\F)<s$, and contain no set of size below $m$.

Let $a=ms-kc-1$, and for $A\in\binom{[n]}a$ write $R=[n]\setminus A$.
Thus, $|R|=n-a=(k+1)c+1=bc+1$.
Define $\X(A)=\F\setminus\mathcal H^k(m,s,\ell;A)$ and $\M(A)=\mathcal H^k(m,s,\ell;A)\setminus\F$.
For a set $E\subseteq[n]$, call $E$ \emph{$A$-extra} if $E\in\X(A)$, and call $E$ \emph{$A$-missing} if $E\in\M(A)$.
Moreover, call $E$ \emph{$A$-canonical} if $E\in\mathcal H^k(m,s,\ell;A)$.
Clearly, every $A$-missing set is $A$-canonical.

For $A\in\binom{[n]}a$, define $g(A)\coloneqq |\F_m\cap\binom Am|$.
Choose $A\in\binom{[n]}a$ maximizing $g(A)$.
We fix this center set $A$ for the rest of the proof.
Since $\F$ contains no set of size below $m$ and $\mathcal H^k(m,s,\ell;A)$ contains every set of size at least $m+L+1$, the family $\X(A)$ is supported on the layers $m,m+1,\ldots,m+L$.
For this fixed set $A$, we have $|\F|-|\mathcal H^k(m,s,\ell;A)|=|\X(A)|-|\M(A)|$.
If $|\X(A)|<|\M(A)|$, then the desired inequality is strict.
Therefore, we may assume
\begin{equation}\label{eq:defect-nonstrict}
        |\X(A)|\ge |\M(A)|
\end{equation}
from now on.

\subsection{Defect of a set and matching parameters}

Given a set $S\subseteq[n]$, define its \emph{$A$-defect} $\sigma(S)$ by
\[
        \sigma(S)\coloneqq m(k+1)-k|S|-|S\cap A|.
\]
By \eqref{eq:Hfamily-def}, $S\in\mathcal H^k(m,s,\ell;A)$ if and only if $\sigma(S)\le0$.
In particular, $S\in\X(A)$ if and only if $S\in\F$ and $\sigma(S)>0$.
Now we prove some basic properties for this definition.

\begin{proposition}
        If $P_1,\ldots,P_s$ form a partition of $[n]$, then $\sum_{r=1}^{s}\sigma(P_r)=1$.
        Moreover, if $S\in\X(A)\cap\binom{[n]}{m+j}$, where $0\le j\le L$, then $\sigma(S)>0$,
        \begin{equation}\label{eq:defect-type-formula}
                |S\cap A|=m-kj-\sigma(S) \text{ and }
                |S\cap R|=(k+1)j+\sigma(S).
        \end{equation}
\end{proposition}

\begin{proof}
        Since $P_1,\ldots,P_s$ partition $[n]$ and $|A|=ms-kc-1$, we have
        \[
                \sum_{r=1}^s\sigma(P_r)
                =sm(k+1)-kn-|A|
                =sm(k+1)-k(ms+c)-(ms-kc-1)=1.
        \]
        If $|S|=m+j$, then $\sigma(S)=m(k+1)-k(m+j)-|S\cap A|=m-kj-|S\cap A|$, and the desired result follows.
\end{proof}

For any matching $\mathcal J\subseteq\X(A)$, define the following parameters.
For $E\in\mathcal J$, write $|E|=m+j(E)$ for $0\le j(E)\le L$.
Since $E\in\X(A)$, $\sigma(E)>0$.
Define the \emph{total layer excess} and the \emph{total $A$-defect} of $\mathcal J$ by
$\omega(\mathcal J)=\sum_{E\in\mathcal J}j(E)$ and
$D(\mathcal J)=\sum_{E\in\mathcal J}\sigma(E)$, respectively.
Thus, $D(\mathcal J)\ge|\mathcal J|$.
Write $V(\mathcal J)=\bigcup_{E\in\mathcal J}E$.
Write $V_R(\mathcal J)=R\cap V(\mathcal J)$ and $V_A(\mathcal J)=A\cap V(\mathcal J)$.
Since the sets in $\mathcal J$ are pairwise disjoint, summing \eqref{eq:defect-type-formula} over $E\in\mathcal J$ gives
\begin{equation}\label{eq:defect-J-support}
        |V_R(\mathcal J)|=b\,\omega(\mathcal J)+D(\mathcal J) \text{ and }
        |V_A(\mathcal J)|=m|\mathcal J|-k\omega(\mathcal J)-D(\mathcal J).
\end{equation}
Since $|R|=bc+1$, every matching $\mathcal J\subseteq\X(A)$ satisfies
\begin{equation}\label{eq:defect-R-capacity}
        D(\mathcal J)+b\omega(\mathcal J)\le bc+1.
\end{equation}
Finally, every member of $\X(A)$ meets $R$: if $S\subseteq A$ and $|S|\ge m$, then $\sigma(S)=m(k+1)-(k+1)|S|\le0$, a contradiction.
Thus, every matching in $\X(A)$ has size at most $|R|=bc+1$.

\subsection{Constants and estimates}

We use $E_k(s,c)$ and $B_k(s,c)$ from
\eqref{eq:defect-proof-E-B} to bound the error terms.  Throughout this
section, all constants implicit in $O_{m,k}(\cdot)$, as well as all
unnamed constants, depend only on $m$ and $k$.

We use the following standard hierarchy convention.  A hierarchy
\[
        x_1\ll x_2\ll\cdots\ll x_r
\]
is chosen from right to left: Each parameter is chosen sufficiently
small as a function of all quantities to its right and of all constants
fixed previously.  Equivalently, when reciprocals occur, the
corresponding constants are chosen sufficiently large.  The hierarchy
is also understood to satisfy any later displayed inequality whose
other quantities depend only on parameters to its right.  Since only
finitely many such inequalities occur, all requirements may be imposed
simultaneously.

Let
$\gamma_{\ref{lem:p-ordered-local-blocker}}$,
$\rho_{\ref{lem:p-ordered-local-blocker}}$, and
$C_{\ref{lem:p-ordered-local-blocker}}$
be the constants from \cref{lem:p-ordered-local-blocker} with $p=m-k$.
As in the remark following that lemma, assume
$0<\rho_{\ref{lem:p-ordered-local-blocker}}\le1$ and
$C_{\ref{lem:p-ordered-local-blocker}}\ge1$.
Set $\xi=1/(2(bL+1))$, choose
$0<\rho\ll \xi,\rho_{\ref{lem:p-ordered-local-blocker}}$, and set
\[
        \alpha_1=
        \min\left\{\frac1{k+1},\frac{\rho}{4b}\right\}.
\]
Fix $\kappa=\frac1{(m-1)!}\left(\frac m4\right)^{m-1}$.
If $s\ge4m$ and $x$ is an integer with $x\ge ms/2$, then
$x-i\ge ms/4$ for every $0\le i\le m-2$, and hence
\begin{equation}\label{eq:defect-kappa-binomial}
        \binom{x}{m-1}\ge\kappa s^{m-1}.
\end{equation}
Set $\chi_m=1/(2^{m-1}(m-1)!)$, choose
$0<\theta\ll\kappa,\chi_m$, and let
$\theta'=\theta/(2(2m)^{m-1})$.
Applying \cref{lem:dense-defect} with $\theta'$ gives constants $\zeta=\zeta(m,\theta')>0$ and $v_{\ref{lem:dense-defect}}
        =v_{\ref{lem:dense-defect}}(m,\theta')$.
Set $\zeta'=\frac14(m/2)^m\zeta$, and, for
$1\le j\le L$ and $0\le x\le\tau_j-1$, write
$q_{j,x}=m+j-x$.

After these quantities have been fixed, choose the remaining constants,
from right to left, according to
\[
        0<
        \frac1{s_0'}
        \ll
        \frac1{\beta'}
        \ll
        \alpha'
        \ll
        \eta
        \ll
        \frac1{M_3}
        \ll
        \frac1{M_2}
        \ll
        \frac1{M_1}
        \ll1.
\]
Here each choice may depend on $m,k$ and on all constants fixed above.
These choices define the constants $\alpha'$, $\beta'$, and $s_0'$
appearing in \cref{thm:reduced-all}.

We first collect the estimates supplied by this hierarchy.

\begin{claim}\label{cl:window-estimates}
        Whenever $s\ge s_0'$ and
        $\beta's^{(k-1)/k}\le c\le \alpha's^{k/(k+1)}$, we have $E_k(s,c)\le\eta s^{m-1}$, $B_k(s,c)\le2s^{m-1}$, $s^{m-k}c^k\ge M_3\eta^{-1}B_k(s,c)$, $cB_k(s,c)\le\eta s^m$ and $c/s\le \eta$.
\end{claim}

\begin{proof}
        Fix $1\le j\le L$ and $0\le x\le\tau_j-1$, and set
        $q=m+j-x-1$.
        Since $x\le m-kj-1$, we have $q\ge(k+1)j$ and
        $x+\frac{k}{k+1}q\le m-1$.  Consequently, the upper bound on $c$
        gives
        \[
                s^xc^q
                \le
                (\alpha')^q s^{x+\frac{k}{k+1}q}
                \le
                (\alpha')^q s^{m-1}.
        \]
        There are only finitely many admissible pairs $(j,x)$.  By the
        hierarchy and the choice of $s_0'$, we also have $s^{-1}\le\eta$
        and $c/s\le\alpha'$.  Thus, the preceding estimate and
        $\alpha'\ll\eta$ give $E_k(s,c)\le\eta s^{m-1}$; since
        $\eta\ll1$, also $B_k(s,c)\le2s^{m-1}$.

        On the other hand, the lower bound on $c$ and the hierarchy give
        \[
                s^{m-k}c^k
                \ge (\beta')^k s^{m-1}
                \ge 2M_3\eta^{-1}s^{m-1}
                \ge M_3\eta^{-1}B_k(s,c).
        \]
        Finally, the hierarchy ensures that $2\alpha'\le\eta$, $c/s\le\alpha's^{-1/(k+1)}\le\alpha'$ and $cB_k(s,c)\le2cs^{m-1}\le\eta s^m$.
        In particular, $c/s\le\eta$.
\end{proof}

\subsection{Closing the comparison}

We first isolate the low-degree part of the center.
Let $H=\F_m\cap\binom Am$.
Define $U=\left\{x\in A:d_H(x)\le \theta s^{m-1}\right\}$ and $A_0=A\setminus U$.
Let $d=|U|$.
Define
\begin{equation}\label{eq:defect-ZU-definition}
        Z_U=\{E\in\binom Am\setminus\F_m:E\cap U\ne\emptyset\}.
\end{equation}
We claim that $Z_U$ is a collection of $A$-missing $m$-sets, and hence $Z_U\subseteq\M(A)$.
Indeed, if $E\in Z_U$, then $E\in\binom Am$ and $E\notin\F_m$.
Since $E\subseteq A$ and $|E|=m$, we have $k|E|+|E\cap A|=km+m=m(k+1)$.
Thus, $E\in\mathcal H^k_m(m,s,\ell;A)$ by \eqref{eq:actual-layer-def}.
As $E$ is an $m$-set not belonging to $\F_m$, it is not in $\F$.
Hence, $E\in\mathcal H^k(m,s,\ell;A)\setminus\F=\M(A)$.

Also define $X_U^\star=\F_m\cap \{P\cup\{z\}:P\in\binom{A_0}{m-1},\ z\in R\}$ and $\Rcal_U=\{E\in\X(A):E\cap U=\emptyset\}$.
Clearly, $X_U^\star\subseteq\Rcal_U$.

The proof is a signed comparison to the family $\mathcal H^k(m,s,\ell;A)$.
The sets in $\X(A)$ contribute positively to $ |\F|-|\mathcal H^k(m,s,\ell;A)| $, while the sets in $\M(A)$ contribute negatively.
Since
\[
        |\F|-|\mathcal H^k(m,s,\ell;A)|=|\X(A)|-|\M(A)|,
\]
it suffices to show that every possible positive contribution forces at least as many negative contributions, with strict inequality unless no $A$-extra set exists.

We first prove an estimate used in the next three structural claims.

\begin{claim}\label{cl:defect-coarse-bounds}
        There is a constant $C_0=C_0(m,k)$ such that $|\X(A)|\le C_0cs^{m-1}$ and $d\le C_0c$.
\end{claim}

\begin{proof}
        We first show that $|\X(A)|\le C_Xcs^{m-1}$ for some constant $C_X=C_X(m,k)$.  In the $m$-layer, every
        $A$-extra set contains a vertex of $R$, and hence there are at most
        $|R|\binom n{m-1}=O_{m,k}(cs^{m-1})$ such sets.  In layer $m+j$,
        where $1\le j\le L$, an $A$-extra set of type $x=|A\cap E|$
        satisfies $0\le x\le\tau_j-1$, and the number of sets of this type
        is $O_{m,k}(s^xc^{q_{j,x}})$.  Since
        $q_{j,x}-1=m+j-x-1\ge(k+1)j$ and
        $x+\frac{k}{k+1}(q_{j,x}-1)\le m-1$, the upper bound on $c$ and
        $\alpha'<1$ give
        \[
                s^xc^{q_{j,x}}
                =c\,s^xc^{q_{j,x}-1}
                \le cs^{m-1}.
        \]
        There are only finitely many admissible pairs $(j,x)$, so the
        asserted bound follows for a suitable $C_X=C_X(m,k)$.

        We now bound $d=|U|$.  Since $c<s$ and $k\le m-1$, we have
        $a=ms-kc-1\ge s-1$, and, for $s_0'$ sufficiently large,
        $\binom{a-1}{m-1}\ge\chi_m s^{m-1}$.  Hence, by the choice of
        $\theta$, each $u\in U$ belongs to at least
        $(\chi_m/2)s^{m-1}$ members of $\binom Am\setminus\F_m$.
        Counting incidences between $U$ and these missing $m$-sets, and
        observing that each member of $Z_U$ is counted at most $m$ times,
        gives
        \[
                |Z_U|\ge\frac{\chi_m}{2m}d s^{m-1}.
        \]
        Since $Z_U\subseteq\M(A)$, \eqref{eq:defect-nonstrict} and the
        first part of the proof yield
        \[
                \frac{\chi_m}{2m}d s^{m-1}
                \le |Z_U|
                \le |\M(A)|
                \le |\X(A)|
                \le C_Xcs^{m-1}.
        \]
        Thus $d=O_{m,k}(c)$.  Choosing $C_0=C_0(m,k)$
        sufficiently large proves the claim.
\end{proof}

The remaining proof needs the following three structural claims.
\cref{cl:defect-capture} shows that all extra sets, except for the star-like family $X_U^\star$, are controlled by a matching and the low-degree part $U$.
\cref{cl:defect-completion} shows that, if there is a sufficiently large matching of extra sets avoiding $U$, then many canonical sets must be missing.
\cref{cl:defect-clean-center} uses the maximality of $g(A)$ to rule out $U\ne\emptyset$ by switching low-degree center vertices with vertices of $R$.

\begin{claim}\label{cl:defect-capture}
        We have $|\X(A)\setminus X_U^\star|\le M_1(\nu(\Rcal_U)+d)E_k(s,c)$.
        Moreover, if $U=\emptyset$, then
        \[
                |\X(A)|\le M_1\nu(\X(A))B_k(s,c).
        \]
\end{claim}
\begin{claim}\label{cl:defect-completion}
        If $\nu(\Rcal_U)>d$, then $|\M(A)\setminus Z_U|\ge M_3(\nu(\Rcal_U)-d)B_k(s,c)$.
\end{claim}
\begin{claim}\label{cl:defect-clean-center}
        We have $U=\emptyset$.
\end{claim}

Now we are ready to derive \cref{thm:reduced-all} from
Claims~\ref{cl:defect-capture}, \ref{cl:defect-completion}
and~\ref{cl:defect-clean-center}.
By \cref{cl:defect-clean-center}, we have $U=\emptyset$.
Hence, $d=0$, $Z_U=\emptyset$, and $\Rcal_U=\X(A)$.
Suppose first that $\X(A)\ne\emptyset$.
Thus, $\nu(\X(A))\ge1$.
Therefore, $\nu(\Rcal_U)=\nu(\X(A))>0=d$.  By
\cref{cl:defect-completion,cl:defect-capture} and $M_3>M_1$,
\[
        |\M(A)|
        \ge M_3\nu(\X(A))B_k(s,c)
        > M_1\nu(\X(A))B_k(s,c)
        \ge |\X(A)|,
\]
contradicting \eqref{eq:defect-nonstrict}.
Hence, $\X(A)=\emptyset$.
Thus, \eqref{eq:defect-nonstrict} gives $\M(A)=\emptyset$, and therefore $\F=\mathcal H^k(m,s,\ell;A)$.
Indeed, $\F$ has no member below size $m$, the family $\mathcal H^k(m,s,\ell;A)$ has no member below size $m$, and every layer above $m+L$ is contained in $\mathcal H^k(m,s,\ell;A)$.
Thus, all possible positive differences are the $A$-extra sets counted by $\X(A)$, and all negative differences are the $A$-missing sets counted by $\M(A)$.

This proves the desired upper bound and the equality statement, assuming Claims~\ref{cl:defect-capture}, \ref{cl:defect-completion} and~\ref{cl:defect-clean-center}.

\subsection{Proof of Claim~\ref{cl:defect-capture}: Capturing A-extra surplus}

Let $\mathcal J\subseteq\Rcal_U$ be a maximum matching.
Consider first the $m$-layer.  If
$F\in(\Rcal_U\setminus X_U^\star)\cap\binom{[n]}m$, then $F$
avoids $U$ and contains a vertex of $R$.  If it
contained exactly one vertex of $R$, its other $m-1$ vertices
would lie in $A_0$, so $F\in X_U^\star$, a contradiction.
Thus every such $F$ contains at least two vertices of $R$.

We next bound degrees in $\Rcal_U\setminus X_U^\star$.  Fix
$v\in[n]$.  In the $m$-layer, the number of members containing
$v$ is $O_{m,k}(cs^{m-2})$: if $v\in R$, choose another vertex
of $R$ and then $m-2$ further vertices; if $v\in A_0$, choose
two vertices of $R$ and then $m-3$ further vertices; and if
$v\in U$, no member is counted.

Now consider layer $m+j$, where $1\le j\le L$, and a type
$x=|F\cap A|$ with $0\le x\le\tau_j-1$.
If $v\in R$, then the number of members of this type containing $v$ is $O_{m,k}(s^xc^{q_{j,x}-1})$.
If $v\in A_0$, then $x\ge1$, and hence the corresponding number is $O_{m,k}(s^{x-1}c^{q_{j,x}})=O_{m,k}(s^xc^{q_{j,x}-1})$.
Again, no member contains a fixed vertex of
$U$.
The displayed quantities are summands of $E_k(s,c)$, and
there are only finitely many admissible types.  Hence there is a constant $C_1=C_1(m,k)$ such that
\begin{equation}\label{eq:defect-local-XU-complement-degree}
        \Delta(\Rcal_U\setminus X_U^\star)
        \le C_1E_k(s,c).
\end{equation}

The same argument bounds the degree of a fixed $u\in U$ in the
whole family $\X(A)$.  In the $m$-layer, an $A$-extra set
containing $u$ must also meet $R$, so there are
$O_{m,k}(cs^{m-2})$ choices.  In layer $m+j$, the condition
$u\in F\cap A$ forces $x\ge1$, and the preceding estimate for a
fixed vertex of $A_0$ applies.  Increasing $C_1$ if necessary,
we therefore have
\begin{equation}\label{eq:defect-local-U-degree}
        d_{\X(A)}(u)\le C_1E_k(s,c)
\end{equation}
for every $u\in U$.

Since $\mathcal J$ is maximum in $\Rcal_U$, every member of
$\Rcal_U$ meets $V(\mathcal J)$.  As every member of
$\mathcal J$ has size at most $m+L$,
\eqref{eq:defect-local-XU-complement-degree} gives
\[
        |\Rcal_U\setminus X_U^\star|
        \le C_1(m+L)\nu(\Rcal_U)E_k(s,c).
\]
Every member of
$(\X(A)\setminus X_U^\star)\setminus\Rcal_U$ meets $U$, so
\eqref{eq:defect-local-U-degree} gives
\[
        |(\X(A)\setminus X_U^\star)\setminus\Rcal_U|
        \le C_1dE_k(s,c).
\]
Therefore,
\[
        |\X(A)\setminus X_U^\star|
        \le
        C_1\bigl((m+L)\nu(\Rcal_U)+d\bigr)E_k(s,c)
        \le
        M_1(\nu(\Rcal_U)+d)E_k(s,c),
\]
by the hierarchy.  This proves the first assertion.

Suppose now that $U=\emptyset$.  Then $A_0=A$ and
$\Rcal_U=\X(A)$.  Let $\mathcal J$ be a maximum matching in
$\X(A)$, and write $T_R=V(\mathcal J)\cap R$ and
$T_A=V(\mathcal J)\cap A$.
Since $\mathcal J$ is maximal, every member of $X_U^\star$
meets $T_A\cup T_R$, and
$|T_A|+|T_R|\le(m+L)\nu(\X(A))$.  The members of
$X_U^\star$ meeting $T_R$, and those avoiding $T_R$ but meeting
$T_A$, number at most
\[
        \begin{aligned}
                |T_R|\binom{|A|}{m-1}
                 & =O_{m,k}\bigl(\nu(\X(A))s^{m-1}\bigr),  \\
                |T_A||R|\binom{|A|}{m-2}
                 & =O_{m,k}\bigl(\nu(\X(A))cs^{m-2}\bigr),
        \end{aligned}
\]
respectively.
The part outside $X_U^\star$ is bounded by the preceding raw
estimate with $d=0$.  Since $cs^{m-2}\le E_k(s,c)$ and
$B_k(s,c)=s^{m-1}+E_k(s,c)$, the hierarchy gives
\[
        |\X(A)|
        \le M_1\nu(\X(A))B_k(s,c).
\]

\subsection{Proof of Claim~\ref{cl:defect-completion}: Completing A-extra matchings}

Assume $\nu(\Rcal_U)>d$, and let $\mathcal J\subseteq\Rcal_U$ be a maximum matching.
Thus, $|\mathcal J|=\nu(\Rcal_U)$.
Write $V_A(\mathcal J)=A\cap V(\mathcal J)$.
Let $\omega=\omega(\mathcal J)$ and $D=D(\mathcal J)$ denote the total layer excess and the total $A$-defect of $\mathcal J$, respectively.
Set
\[
        h=\min\left\{\lfloor\rho c\rfloor,
        \left\lfloor\frac{D-d-1}{k}\right\rfloor\right\},
        \qquad
        u=c-\omega-h,
        \qquad
        N=A_0\setminus V_A(\mathcal J).
\]
Define $t=s-\nu(\Rcal_U)-u$ and
$\lambda=pu+D-d-1-kh$.
Equivalently,
$t=s-c-\nu(\Rcal_U)+\omega+h$ and
$\lambda=pc-1-d+D-p\omega-mh$.
These definitions encode the two balances used in the completion:
$|\mathcal J|+t+u=s$ and
$\lambda-pu=D-d-1-kh$.
After keeping $\mathcal J$, a perfect matching in $H[N\setminus T]$
will contribute $t$ $m$-sets, while the remaining vertices will be
divided into $u$ mixed blocks, $h$ of which have size $b+1$ on the
second side.

The following claim collects the conditions needed for this completion
attempt.
\begin{claim}\label{cl:completion-bookkeeping}
        With the notation above, the following statements hold.
        \begin{enumerate}[label=\textup{(\roman*)},leftmargin=2.2em]
                \item
                      We have $\omega\le(1-\xi)c$.

                \item
                      We have $u\ge \xi c/2$, $h\le u$,
                      $h\le
                              \rho_{\ref{lem:p-ordered-local-blocker}}u$ and
                      $u\ge C_{\ref{lem:p-ordered-local-blocker}}$.

                \item
                      We have
                      $h+1\ge\alpha_1(\nu(\Rcal_U)-d)$.

                \item
                      We have $|N|=mt+\lambda$ and
                      $\lambda-pu=D-d-1-kh\ge0$.
                      Moreover, $0\le\lambda\le(m+1)c$,
                      $|t-s|\le(b+4)c$, and
                      $s/2\le t\le3s/2$.
        \end{enumerate}
\end{claim}

\begin{proof}
        Since each $\sigma(E)$ is a positive integer, we have
        $D\ge|\mathcal J|=\nu(\Rcal_U)$.
        Also, \eqref{eq:defect-R-capacity} gives
        $D+b\omega\le bc+1$.
        Since every nonzero contribution $j(E)$ to $\omega$ is at most
        $L$, we have
        $\omega\le L|\{E\in\mathcal J:j(E)\ge1\}|$, whereas
        $D\ge|\{E\in\mathcal J:j(E)\ge1\}|$.
        Thus, $D\ge\omega/L$, and consequently
        \[
                \omega
                \le\frac{bc+1}{b+1/L}
                =\frac{bL}{bL+1}c+\frac{L}{bL+1}
                \le(1-\xi)c.
        \]
        Here the last inequality uses $c\ge2L$ and
        $\xi=1/(2(bL+1))$.
        This proves (i).

        Since $D\ge\nu(\Rcal_U)>d$ and all these quantities are
        integers, we have $D-d-1\ge0$.
        Hence, $0\le h\le\rho c$.
        By (i), $u=c-\omega-h\ge(\xi-\rho)c$.
        By the choice of $\rho$, we have
        $u\ge\xi c/2$,
        $h\le\rho c\le\xi c/4\le u$, and
        \[
                \frac hu
                \le\frac{\rho}{\xi-\rho}
                \le\rho_{\ref{lem:p-ordered-local-blocker}}.
        \]
        Moreover, the lower bound on $c$ and the choice of $\beta'$ in
        the hierarchy give
        $u\ge\xi c/2\ge
                C_{\ref{lem:p-ordered-local-blocker}}$.
        This proves (ii).

        If the minimum in the definition of $h$ is attained by the second term, then
        \[
                h+1
                =
                \left\lceil\frac{D-d}{k}\right\rceil
                \ge\frac{D-d}{k+1}
                \ge\frac{\nu(\Rcal_U)-d}{k+1}.
        \]
        If it is attained by the first term, then $h+1>\rho c$.
        Since every member of $\Rcal_U$ meets $R$, we have
        $\nu(\Rcal_U)\le|R|=bc+1\le2bc$, and hence
        $h+1\ge\frac{\rho}{4b}(\nu(\Rcal_U)-d)$.
        Both cases give
        $h+1\ge\alpha_1(\nu(\Rcal_U)-d)$ by the choice of $\alpha_1$.
        This proves (iii).

        By the definition of $h$,
        $\lambda-pu=D-d-1-kh\ge0$.
        In particular, $\lambda\ge pu\ge0$.
        Since $\mathcal J\subseteq\Rcal_U$, every member of
        $\mathcal J$ avoids $U$, and hence
        $V_A(\mathcal J)\subseteq A_0$.
        By \eqref{eq:defect-J-support},
        $|V_A(\mathcal J)|=m\nu(\Rcal_U)-k\omega-D$.
        Using $m-p=k$ and $u=c-\omega-h$, we obtain
        \[
                \begin{aligned}
                        |N|
                         & =
                        a-d-|V_A(\mathcal J)|=
                        ms-kc-1-d-m\nu(\Rcal_U)+k\omega+D \\
                         & =m(s-\nu(\Rcal_U)-u)
                        +pu+D-d-1-kh=mt+\lambda.
                \end{aligned}
        \]

        From $D+b\omega\le bc+1$ and
        $c-\omega=u+h$, we have
        \[
                \lambda\le
                pu+b(c-\omega)-d-kh=
                (m+1)u+h-d\le
                (m+1)(u+h)
                \le(m+1)c.
        \]
        Finally, the definition of $t$ gives
        $s-t=\nu(\Rcal_U)+u$.
        Since $\nu(\Rcal_U)\le bc+1\le(b+1)c$ and $u\le c$, we have
        $0\le s-t\le(b+2)c$, and hence
        $|t-s|\le(b+4)c$.
        By \cref{cl:window-estimates}, $c/s\le\eta$.
        Since $\eta$ is sufficiently small, $(b+4)c\le s/2$, so
        $s/2\le t\le s\le3s/2$.
        This proves (iv).
\end{proof}

We now collect some consequences of \cref{cl:completion-bookkeeping}.
By the convention in the remark after \cref{lem:p-ordered-local-blocker}, we may assume that $C_{\ref{lem:p-ordered-local-blocker}}\ge1$.
Hence, \cref{cl:completion-bookkeeping} gives $u\ge1$.
Thus, we have
\begin{equation}\label{eq:completion-feasibility}
        p\le pu\le\lambda\le |N|-m \text{ and }
        |N|\ge ms/2.
\end{equation}
Indeed, the first two inequalities follow from $u\ge1$ and \cref{cl:completion-bookkeeping}(iv).
Also $t\ge s/2\ge1$, since $s_0'$ is sufficiently large.
Since $|N|=mt+\lambda$, we have $\lambda\le |N|-m$ and $|N|\ge mt\ge ms/2$.

We now consider the $\lambda$-subsets $T\subseteq N$.
A set $T$ is called \emph{bad} if the family $H[N\setminus T]$ has no perfect matching, and \emph{good} otherwise.

\medskip\noindent\textbf{Case 1.} At least half of these $\lambda$-subsets of $N$ are bad.
\medskip

Fix $T\in\binom N\lambda$ and $x\in N\setminus T$.
Since $N\subseteq A_0$, the definition of $U$ gives $d_H(x)>\theta s^{m-1}$.
The vertices deleted from $A$ before passing to $N\setminus T$
lie in $U\cup V_A(\mathcal J)\cup T$.
By \cref{cl:defect-coarse-bounds,cl:completion-bookkeeping}, we conclude that $|U|=d=O_{m,k}(c)$, $|V_A(\mathcal J)|\le m\nu(\Rcal_U)\le m(b+1)c$ and $|T|=\lambda\le(m+1)c$.
Hence the total number of deleted vertices is $O_{m,k}(c)$.
For a fixed deleted vertex $y\ne x$, the number of $m$-sets containing
both $x$ and $y$ is $O_m(s^{m-2})$.  Thus deleting all these vertices
removes from the degree of $x$ at most $O_{m,k}(cs^{m-2})
        \le{\theta}s^{m-1}/2$,
where the last inequality follows from $c/s\le\eta$ in
\cref{cl:window-estimates} and the choice of $\eta$ in the hierarchy.
Therefore, by \cref{cl:completion-bookkeeping}(iv), we have $|N| = mt+\lambda$ and $s/2\le t\le 3s/2$.
Hence, for every $T\in\binom{N}{\lambda}$, $|N\setminus T|
        = |N|-\lambda
        = mt$.
Therefore,
\[
        \delta\bigl(H[N\setminus T]\bigr)
        \ge \frac{\theta}{2}s^{m-1}
        \ge \theta'|N\setminus T|^{m-1},
\]
where the last inequality follows from $|N\setminus T|=mt\le 2ms$ and the definition $\theta'=\theta/(2(2m)^{m-1})$.
In particular, $m\mid |N\setminus T|$. Moreover, since
$t\ge s/2$ and $s_0'$ was chosen sufficiently large, we have $|N\setminus T|=mt\ge v_{\ref{lem:dense-defect}}$.
Since $T$ is bad, $H[N\setminus T]$ has no perfect matching.
Therefore, \cref{lem:dense-defect} gives
\[
        \left|\binom{N\setminus T}m\setminus H\right|
        \ge \zeta|N\setminus T|^m
        \ge 4\zeta's^m
\]
for each bad $T$.
Double-counting pairs $(T,E)$ with $T$ bad and $E\in\binom{N\setminus T}m\setminus H$ yields
\begin{equation}\label{eq:defect-bad-missing}
        |\M(A)\setminus Z_U|\ge \zeta's^m.
\end{equation}
Indeed, each bad $T$ contributes at least $4\zeta's^m$ missing $m$-sets, and at least half of the $\lambda$-subsets of $N$ are bad.
Hence, the number of such pairs is at least $\frac12\binom{|N|}{\lambda}\cdot 4\zeta's^m$.
On the other hand, \eqref{eq:completion-feasibility} gives $\lambda\le |N|-m$.
A fixed $m$-set $E\subseteq N$ can occur only when $T\cap E=\emptyset$, and hence for at most $\binom{|N|-m}{\lambda}$ choices of $T$.
Since $\binom{|N|}{\lambda}/\binom{|N|-m}{\lambda}=\prod_{r=0}^{m-1}(|N|-r)/(|N|-\lambda-r)\ge1$, dividing the lower count for pairs by this upper multiplicity gives at least $2\zeta's^m$ missing $m$-sets, and we keep the weaker bound in \eqref{eq:defect-bad-missing}.
Since $N\subseteq A_0\subseteq A$, every counted $E$ is an $A$-missing $m$-set, and since $E\cap U=\emptyset$ it is not counted by $Z_U$.
Since every member of $\Rcal_U$ meets $R$, every matching in $\Rcal_U$ uses distinct vertices of $R$.
Hence, $\nu(\Rcal_U)\le |R|=bc+1$.
As $d\ge0$ and $c\ge1$, we have
\[
        M_3(\nu(\Rcal_U)-d)B_k(s,c)
        \le M_3(b+1)cB_k(s,c)
        \le M_3(b+1)\eta s^m,
\]
where the last inequality follows from \cref{cl:window-estimates}.
Since $\eta$ is chosen after $M_3$ and $\zeta'$ is fixed, the
hierarchy gives
\[
        M_3(b+1)\eta\le\zeta'.
\]
Thus, \cref{cl:defect-completion} follows in this case.

\medskip\noindent\textbf{Case 2.} More than half of these $\lambda$-subsets of $N$ are good.
\medskip

For each good $T\in\binom N\lambda$, fix once and for all a perfect matching in $H[N\setminus T]$.
Since $pu\le\lambda$ by \eqref{eq:completion-feasibility}, choose a set $P_T\subseteq T$ with $|P_T|=pu$.
Define $B_R=R\setminus V_R(\mathcal J)$ and $C=(U\cup T\cup B_R)\setminus P_T$.
Then $P_T\cap C=\emptyset$, $P_T\cup C=U\cup T\cup B_R$, and
\[
        \begin{aligned}
                |C| & =
                |U|+|T|+|B_R|-|P_T| =
                d+\lambda+bc+1-D-b\omega-pu \\
                    & =
                (m+1)(c-\omega)-mh-pu =
                b(c-\omega-h)+h
                =
                bu+h.
        \end{aligned}
\]
Here the first equality uses $P_T\subseteq T$ and the disjointness of
$U$, $T$ and $B_R$.
The second uses $|U|=d$, $|T|=\lambda$,
$|P_T|=pu$ and
\[
        |B_R|=|R|-|V_R(\mathcal J)|=bc+1-D-b\omega.
\]
The remaining equalities follow from the definition of $\lambda$, $p+b=m+1$, and $u=c-\omega-h$.

For a good $T$, completing $\mathcal J$ and the fixed perfect matching in $H[N\setminus T]$ to an $s$-matching reduces to finding a suitable mixed ordered partition on the remaining vertices.
For each $e\in\binom{P_T}p$, define
\[
        \G_e^{(b)}=\{Y\in\binom Cb:e\cup Y\in\F_{m+1}\} \text{ and }
        \G_e^{(b+1)}=\{Y\in\binom C{b+1}:e\cup Y\in\F_{m+2}\}.
\]
In this setting, a \emph{mixed ordered partition} consists of ordered partitions
$P_T=e_1\sqcup\cdots\sqcup e_u$ and
$C=Y_1\sqcup\cdots\sqcup Y_u$, where each $e_r$ has size $p$,
exactly $h$ of the sets $Y_r$ have size $b+1$, the remaining
$u-h$ have size $b$, and $Y_r\in\G_{e_r}^{(|Y_r|)}$ for every
$r\in[u]$.
Such a mixed ordered partition would combine with $\mathcal J$ and
the fixed perfect $t$-matching of $H[N\setminus T]$ as follows.
The vertex set $[n]$ satisfies
\[
        [n]
        =V_A(\mathcal J)\sqcup V_R(\mathcal J)
        \sqcup (N\setminus T)\sqcup P_T\sqcup C.
\]
Indeed, $A=V_A(\mathcal J)\sqcup N\sqcup U$ and $R=V_R(\mathcal J)\sqcup B_R$, while $N=(N\setminus T)\sqcup T$ and $T\cup U\cup B_R=P_T\sqcup C$ by the definitions of $P_T$ and $C$ above.
Thus $\mathcal J$ covers $V_A(\mathcal J)\cup V_R(\mathcal J)$, the perfect $t$-matching in $H[N\setminus T]$ covers $N\setminus T$, and the mixed ordered partition covers $P_T\cup C$.
Every member used belongs to $\F$: the members of $\mathcal J$ lie in
$\Rcal_U\subseteq\X(A)\subseteq\F$, the perfect matching on
$N\setminus T$ lies in $H\subseteq\F_m$, and the sets
$e_r\cup Y_r$ coming from the mixed ordered partition lie in
$\F_{m+1}$ or $\F_{m+2}$ by the definitions of
$\G_e^{(b)}$ and $\G_e^{(b+1)}$.
The number of resulting members is
\[
        |\mathcal J|+t+u
        =\nu(\Rcal_U)+s-c-\nu(\Rcal_U)+\omega+h+c-\omega-h
        =s.
\]
This would give an $s$-matching in $\F$, a contradiction.
All hypotheses of \cref{lem:p-ordered-local-blocker} have been verified in \cref{cl:completion-bookkeeping} and the paragraph above, so \cref{lem:p-ordered-local-blocker} gives
\begin{equation}\label{eq:defect-ordered-blocker-output}
        \sum_{e\in\binom{P_T}p}
        \left(
        \left|\binom Cb\setminus\G_e^{(b)}\right|
        +
        \left|\binom C{b+1}\setminus\G_e^{(b+1)}\right|
        \right)
        \ge \gamma_{\ref{lem:p-ordered-local-blocker}}(h+1)u^m.
\end{equation}
Since $u\ge\xi c/2$, the right-hand side of
\eqref{eq:defect-ordered-blocker-output} is at least
$\gamma_{\ref{lem:p-ordered-local-blocker}}(\xi/2)^m(h+1)c^m$.

Every pair $(e,Y)$ counted in \eqref{eq:defect-ordered-blocker-output} gives a set $W=e\cup Y$ absent from $\F$, by the definitions of $\G_e^{(b)}$ and $\G_e^{(b+1)}$.
Since $e\subseteq P_T\subseteq T\subseteq N\subseteq A_0\subseteq A$, the set $W$ contains at least $p$ vertices of $A$.

If $|Y|=b$, then $|W|=p+b=m+1$ and $|W\cap A|\ge p=m-k=\tau_1$, so $W\in\mathcal H^k_{m+1}(m,s,\ell;A)$.
If $|Y|=b+1$, then $|W|=m+2$.
When $L\ge 2$, the $(m+2)$-layer has threshold $\tau_2=m-2k\le m-k=p\le |W\cap A|$, and hence $W\in\mathcal H^k_{m+2}(m,s,\ell;A)$.
When $L=1$, we have $m+2=m+L+1$, so the entire $(m+2)$-layer is
canonical. Thus, in either case, $W\in
        \mathcal H^k(m,s,\ell;A)\setminus\mathcal F
        =\mathcal M(A)$.
Since $|W|=m+1$ or $m+2$, it is not counted by $Z_U$.

We now double-count triples $(T,e,Y)$, where $T\in\binom N\lambda$ is good and $(e,Y)$ is counted in \eqref{eq:defect-ordered-blocker-output}.
Since at least half of the $\lambda$-subsets of $N$ are good, the number of such triples is at least
\[
        \frac12\binom{|N|}{\lambda}
        \gamma_{\ref{lem:p-ordered-local-blocker}}
        \left(\frac\xi2\right)^m
        (h+1)c^m.
\]
Conversely, fix $W\in\M(A)\setminus Z_U$.
If $W=e\cup Y$ arises from such a triple, then $e\subseteq W\cap T$.
There are at most $\binom{m+2}p$ choices for $e\subseteq W$.
After $e$ is fixed, the number of $\lambda$-subsets $T\subseteq N$ containing $e$ is at most $\binom{|N|-p}{\lambda-p}$, and $Y$ is then determined by $W$ and $e$.
Hence each fixed $W\in\M(A)\setminus Z_U$ is counted at most $\binom{m+2}p\binom{|N|-p}{\lambda-p}$ times.
Therefore,
\begin{equation}\label{eq:defect-good-double-count-raw}
        |\M(A)\setminus Z_U|
        \ge
        \frac{\gamma_{\ref{lem:p-ordered-local-blocker}}}
        {2\binom{m+2}{p}}
        \left(\frac\xi2\right)^m
        (h+1)c^m
        \frac{\binom{|N|}{\lambda}}
        {\binom{|N|-p}{\lambda-p}}.
\end{equation}
By \eqref{eq:completion-feasibility}, $\lambda\ge p$ and
$|N|\ge ms/2$, while \cref{cl:completion-bookkeeping} gives
$\lambda\le(m+1)c$.  Since $s_0'$ is sufficiently large, for every
$0\le i\le p-1$ we have $|N|-i\ge ms/4$.  Consequently,
\begin{equation}\label{eq:defect-binomial-ratio-lower}
        \frac{\binom{|N|}{\lambda}}
        {\binom{|N|-p}{\lambda-p}}
        =
        \prod_{i=0}^{p-1}\frac{|N|-i}{\lambda-i}
        \ge
        \gamma_1\left(\frac sc\right)^p
\end{equation}
for some constant $\gamma_1=\gamma_1(m,k)>0$.
Combining \eqref{eq:defect-good-double-count-raw} and
\eqref{eq:defect-binomial-ratio-lower}, and then applying
\cref{cl:completion-bookkeeping,cl:window-estimates}, we obtain, for some
constant $\gamma=\gamma(m,k)>0$,
\[
        |\M(A)\setminus Z_U|\ge \gamma(h+1)s^{m-k}c^k\ge \gamma\alpha_1M_3\eta^{-1}
        (\nu(\Rcal_U)-d)B_k(s,c)  \ge M_3(\nu(\Rcal_U)-d)B_k(s,c).
\]
The last inequality follows because $\gamma$ and $\alpha_1$ depend only
on quantities fixed before $\eta$ is chosen, so the hierarchy ensures
$\gamma\alpha_1\eta^{-1}\ge1$.
This proves \cref{cl:defect-completion}.
\subsection{Proof of Claim~\ref{cl:defect-clean-center}: Cleaning the center}

We prove \cref{cl:defect-clean-center} by a center switch.
For $S\subseteq R$, define
\[
        T(S)=\{P\cup\{z\}:P\in\binom{A_0}{m-1},\ z\in S\}.
\]
For $U'\subseteq U$ and $S\subseteq R$ with
$|U'|=|S|=e\ge1$, we say that the pair $(U',S)$ is
\emph{switch-admissible} if $d\le M_2e$ and
\begin{equation}\label{eq:defect-switch-admissible}
        |\X(A)\setminus T(S)|
        \le
        |\M(A)\setminus Z_U|+M_2dE_k(s,c).
\end{equation}
The switch replaces low-degree vertices of $A$ by vertices of $R$.
Removing vertices from $U$ loses few members of $H=\F_m\cap\binom Am$;
choosing vertices of $R$ that capture enough extra sets of the form
$P\cup\{z\}$ increases $g$, contradicting the maximality of $A$.
We first find a switch-admissible pair and then perform the switch.
Thus, suppose for a contradiction that $U\ne\emptyset$ from now on.
Since $d=|U|$, we have $d>0$.

\begin{claim}\label{cl:defect-switch-admissible-exists}
        There exist $U'\subseteq U$ and
        $S\subseteq R$ such that $(U',S)$ is switch-admissible.
\end{claim}

\begin{proof}
        Let $\mu=\nu(X_U^\star)$, and choose a maximum matching $\mathcal J_0
                =
                \{P_r\cup\{z_r\}:1\le r\le\mu\}
                \subseteq X_U^\star$.
        Let $S_\mu=\{z_1,\ldots,z_\mu\}$.
        Since $X_U^\star\subseteq\Rcal_U$, we have
        $\mu\le\nu(\Rcal_U)$.
        We split the proof into two cases based on the value of $d$.

        \medskip
        \noindent\textbf{Case 1.} Suppose that $d>|R|$.
        \medskip

        In this case, choose $U'\subseteq U$ with $|U'|=|R|$, and take $S=R$.
        Thus, $X_U^\star\subseteq T(S)$, because every member of
        $X_U^\star$ has the form $P\cup\{z\}$ with
        $P\in\binom{A_0}{m-1}$ and $z\in R=S$.

        Every member of $\Rcal_U$ is $A$-extra, and hence meets $R$.
        Thus, any matching in $\Rcal_U$ uses distinct vertices of $R$,
        so $\nu(\Rcal_U)\le |R|<d$.
        By \cref{cl:defect-capture},
        \[
                |\X(A)\setminus T(S)|
                \le
                |\X(A)\setminus X_U^\star|
                \le
                M_1(\nu(\Rcal_U)+d)E_k(s,c)
                \le
                2M_1dE_k(s,c)\le M_2dE_k(s,c).
        \]
        The last inequality follows from the hierarchy.  Hence,
        \eqref{eq:defect-switch-admissible} holds.

        It remains to verify the size condition.
        Here $|U'|=|S|=|R|\ge1$.
        By \cref{cl:defect-coarse-bounds}, $d\le C_0c$.
        Since $|R|=bc+1\ge c$ and $M_2$ is chosen sufficiently large
        in the hierarchy, $d\le C_0c\le M_2|R|$.
        Thus, $(U',S)$ is switch-admissible in this case.

        \medskip
        \noindent\textbf{Case 2.} Suppose that $d\le |R|$.
        \medskip

        In this case, let $U'=U$.
        Choose $S\subseteq R$ with $|S|=d$ so that $|S\cap S_\mu|$ is
        maximal.
        Given a number $x$, write $x_+$ for $\max\{x,0\}$.
        Thus, $|S_\mu\setminus S|=(\mu-d)_+$ since $|S|=d$ and $|S_\mu|=\mu$ by definition.
        We first bound the size of $X_U^\star\setminus T(S)$.
        \begin{claim}\label{cl:clean-center-XU-loss}
                There is a constant $C_7=C_7(m,k)>0$ such that
                \[
                        |X_U^\star\setminus T(S)|
                        \le
                        (\mu-d)_+\binom{|A_0|}{m-1}
                        +C_7\mu cs^{m-2}.
                \]
        \end{claim}

        \begin{proof}
                Every member of $X_U^\star$ has the form $P\cup\{z\}$,
                where $P\in\binom{A_0}{m-1}$ and $z\in R$.
                If $z\in S$, then this set belongs to $T(S)$.
                Hence only vertices $z\in R\setminus S$ can contribute to
                $X_U^\star\setminus T(S)$.
                Since $S$ maximizes $|S\cap S_\mu|$ among all
                $d$-subsets of $R$, the number of vertices of $S_\mu$
                left outside $S$ is exactly $(\mu-d)_+$.
                The members with $z\in S_\mu\setminus S$ therefore
                contribute at most $(\mu-d)_+\binom{|A_0|}{m-1}$.

                It remains to count members
                $P\cup\{z\}\in X_U^\star\setminus T(S)$ with
                $z\notin S_\mu$.
                Let $Q=\bigcup_{r=1}^{\mu}P_r$.
                If such a member had $P\cap Q=\emptyset$, then
                $P\cup\{z\}$ would be disjoint from every member
                $P_r\cup\{z_r\}$ of $\mathcal J_0$: the sets $P$ and
                $P_r$ are disjoint by $P\cap Q=\emptyset$, the vertex
                $z$ is distinct from every $z_r$ because $z\notin S_\mu$,
                and $P\subseteq A_0$ while each $z_r\in R$.
                This would enlarge the matching $\mathcal J_0$,
                contradicting its maximality.
                Thus every remaining member has $P\cap Q\ne\emptyset$.

                We count these remaining members by first choosing
                $z\in R$, then a vertex of $P\cap Q$, and then the other
                $m-2$ vertices of $P$ from $A_0$.
                Since $|R|\le(b+1)c$, $|Q|\le(m-1)\mu$, and
                $|A_0|<(m+1)s$, their number is
                $O_{m,k}(\mu cs^{m-2})$.  Choosing $C_7=C_7(m,k)$
                sufficiently large proves the claim.
        \end{proof}

        We now combine this estimate with \cref{cl:defect-capture}.
        Since $\binom{|A_0|}{m-1}=O_{m,k}(s^{m-1})$,
        $\mu\le\nu(\Rcal_U)$, $cs^{m-2}\le E_k(s,c)$, and
        $B_k(s,c)=s^{m-1}+E_k(s,c)$, there is a constant
        $C_8=C_8(m,k,M_1)$ such that
        \[
                \begin{aligned}
                        |\X(A)\setminus T(S)|
                         & \le
                        |\X(A)\setminus X_U^\star|
                        +|X_U^\star\setminus T(S)| \\
                         & \le
                        C_8(\nu(\Rcal_U)-d)_+B_k(s,c)
                        +C_8dE_k(s,c)              \\
                         & \le
                        C_8(\nu(\Rcal_U)-d)_+B_k(s,c)
                        +M_2dE_k(s,c),
                \end{aligned}
        \]
        where the last inequality follows from the hierarchy.

        If $\nu(\Rcal_U)\le d$, then the first term vanishes, and
        \eqref{eq:defect-switch-admissible} follows immediately.
        If $\nu(\Rcal_U)>d$, then \cref{cl:defect-completion} gives
        $|\M(A)\setminus Z_U|\ge
                M_3(\nu(\Rcal_U)-d)B_k(s,c)$.
        Since $M_3$ is chosen after $M_1$ and $M_2$, the hierarchy
        ensures that $M_3>C_8$, and this again implies
        \eqref{eq:defect-switch-admissible}.

        Finally, $|U'|=|S|=d\ge1$ and
        $d\le M_2d=M_2|U'|$ are immediate in this case.
        Thus, $(U',S)$ is switch-admissible in Case 2.
        The two cases prove the claim.
\end{proof}

By \cref{cl:defect-switch-admissible-exists}, choose
$U'\subseteq U$ and $S\subseteq R$ such that $(U',S)$ is switch-admissible.
Write $e=|U'|=|S|$ and let $A^*=(A\setminus U')\cup S$.
Clearly, $A^*$ is again a set of size $a$.
Define $L(U')=\{P\cup\{u\}:P\in\binom{A_0}{m-1},\ u\in U'\}$.
We now show that this switch increases the value of $g$.

\begin{claim}\label{cl:defect-center-switch}
        We have $g(A^*)>g(A)$.
\end{claim}

\begin{proof}
        Recall that $A_0=A\setminus U$ and $d=|U|$.
        Since $U\subseteq A$, we have $|A_0|=a-d$.
        By \cref{cl:defect-coarse-bounds}, $d\le C_0c$, and by
        \cref{cl:window-estimates}, $c/s\le\eta$.
        Since $|A_0|=a-d$ and $a=ms-kc-1$, we have
        \[
                a-d
                \ge ms-(k+C_0)c-1
                \ge\bigl(m-(k+C_0)\eta\bigr)s-1
                \ge\frac{ms}{2},
        \]
        by the hierarchy and the choice of $s_0'$.
        Therefore, \eqref{eq:defect-kappa-binomial} gives
        $\binom{a-d}{m-1}\ge \kappa s^{m-1}$.

        Every member of $T(S)$ has the form $P\cup\{z\}$, where
        $P\subseteq A_0\subseteq A\setminus U'$ and $z\in S\subseteq R$.
        Hence, it is an $m$-set contained in $A^*$, but it is not
        contained in $A$. Thus, every member of $T(S)$ lies in
        $\binom{A^*}m\setminus\binom Am$.
        Since every member of $T(S)$ has size $m$, every member of
        $T(S)\cap\X(A)$ belongs to $\F_m$ and is counted by $g(A^*)$.
        Conversely, every member counted by $g(A)$ but not by $g(A^*)$
        is a member of $H=\F_m\cap\binom Am$ that meets $U'$.
        Therefore,
        \begin{equation}\label{eq:g-switch-basic-gain-loss}
                g(A^*)-g(A)
                \ge
                |T(S)\cap\X(A)|
                -|\{E\in H:E\cap U'\ne\emptyset\}|.
        \end{equation}

        Since $U'\subseteq U$, the definition of $U$ gives
        $d_H(u)\le \theta s^{m-1}$ for every $u\in U'$. Hence,
        \begin{equation}\label{eq:g-switch-lost-H-bound}
                |\{E\in H:E\cap U'\ne\emptyset\}|
                \le
                \sum_{u\in U'}d_H(u)
                \le e\theta s^{m-1}.
        \end{equation}

        We next lower bound $|T(S)\cap\X(A)|$.
        For each $u\in U'$, the family $L(U')$ contains exactly
        $\binom{a-d}{m-1}$ sets containing $u$, and at most
        $\theta s^{m-1}$ of them lie in $H$.
        Since each set in $L(U')$ contains a unique vertex of $U'$,
        \[
                |L(U')\setminus\F_m|
                \ge
                e\binom{a-d}{m-1}-e\theta s^{m-1}.
        \]
        Every set in $L(U')\setminus\F_m$ belongs to $Z_U$ by
        \eqref{eq:defect-ZU-definition}, because it is an $m$-subset of
        $A$, is absent from $\F_m$, and meets $U$. Hence,
        \begin{equation}\label{eq:g-switch-ZU-lower}
                |Z_U|
                \ge
                e\binom{a-d}{m-1}-e\theta s^{m-1}.
        \end{equation}

        Since $(U',S)$ is switch-admissible, \eqref{eq:defect-switch-admissible}
        holds. Using this, \eqref{eq:defect-nonstrict}, and
        $Z_U\subseteq\M(A)$, we obtain
        \begin{equation}\label{eq:g-switch-XT-lower}
                \begin{aligned}
                        |\X(A)\cap T(S)|
                         & =
                        |\X(A)|-|\X(A)\setminus T(S)|\ge
                        |\M(A)|-|\M(A)\setminus Z_U|-M_2dE_k(s,c) \\
                         & =
                        |Z_U|-M_2dE_k(s,c) \ge
                        e\binom{a-d}{m-1}-e\theta s^{m-1}
                        -M_2dE_k(s,c),
                \end{aligned}
        \end{equation}
        where the last inequality follows from \eqref{eq:g-switch-ZU-lower}.

        Finally, \eqref{eq:g-switch-basic-gain-loss},
        \eqref{eq:g-switch-lost-H-bound}, and
        \eqref{eq:g-switch-XT-lower} give
        \[
                g(A^*)-g(A)\ge e\binom{a-d}{m-1}
                -2e\theta s^{m-1}-M_2dE_k(s,c)\ge e(\kappa-2\theta-M_2^2\eta)s^{m-1}>0.
        \]
        Here the second inequality uses $d\le M_2e$ from
        switch-admissibility and $E_k(s,c)\le\eta s^{m-1}$ from
        \cref{cl:window-estimates}; the final inequality follows from
        $2\theta+M_2^2\eta<\kappa$, as ensured by the hierarchy.
        Hence $g(A^*)>g(A)$.
\end{proof}
This contradicts the choice of $A$.
Therefore, $U=\emptyset$.
This completes the proof of \cref{cl:defect-clean-center}.

\section*{Declaration on the use of AI}
The authors used generative AI for language polishing and to assist in the search for suitable parameter choices for the application of \cref{lem:p-ordered-local-blocker}.
The authors independently verified the parameter choices and all mathematical arguments and take full responsibility for the content of the manuscript.

\end{document}